\documentclass{amsart}
\usepackage{etex}
\usepackage[T1]{fontenc}
\usepackage[english]{babel}
\usepackage{mathtools}
\usepackage[margin=1.45in]{geometry}

\usepackage{textcomp}
\usepackage{enumerate}
\usepackage{enumitem}
\usepackage{microtype}
\usepackage{lmodern}
\usepackage{graphicx}
\usepackage[numbers]{natbib}

\usepackage[text]{esdiff}
\usepackage{scalerel}

\usepackage{amsfonts}
\usepackage{amsmath}
\usepackage{amstext}
\usepackage{amssymb}
\usepackage{amsthm}
\usepackage{stmaryrd}

\usepackage{tikz}
\usepackage{tikz-cd}
\usepackage{units}
\usepackage{hyperref}
\usepackage{color}
\usepackage{scalerel}
\usepackage{stackengine}

\usepackage{stackrel}

\newcommand{\nvbzero}[2]{
\mathbf{0}^{#1}_{#2}
}
\newcommand{\dvtimes}[2]{
\stackrel[{\scriptscriptstyle #2}]{{\scriptscriptstyle #1}}{\cdot}
}
\newcommand{\dvplus}[2]{
\stackrel[{\scriptscriptstyle #2}]{{\scriptscriptstyle #1}}{+}
}
\newcommand{\dvminus}[2]{
\stackrel[{\scriptscriptstyle #2}]{{\scriptscriptstyle #1}}{-}
}

\newcommand{\infVB}{\infty\text{-VB}}

\newcommand{\lie}[1]{\mathfrak{#1}}

\DeclareMathOperator {\pr}    {pr}

\DeclareMathOperator {\rk}    {rk}

\DeclareMathOperator {\Hom}   {Hom}
\DeclareMathOperator {\id}    {id}
\DeclareMathOperator {\Mor}    {Mor}
\newcommand{\upp}{{\mathbf{up}}}
\newcommand{\low}{{\mathbf{low}}}

\newcommand{\R}{\mathbb R}
\newtheorem{thm}{Theorem}[section]
\newtheorem{mydef}[thm]{Definition}
\newtheorem{lemma}[thm]{Lemma}
\newtheorem{cor}[thm]{Corollary}
\newtheorem{prop}[thm]{Proposition}
\newtheorem{rmk}[thm]{Remark}
\newtheorem{example}[thm]{Example}

\newcommand{\inv}{^{-1}}
\newcommand{\nset}{{\underline{n}}}
\newcommand{\N}{\mathbb N}
\newcommand{\E}{\mathbb E}
\newcommand{\F}{\mathbb F}
\newcommand{\A}{\mathcal A}
\renewcommand{\S}{\mathcal S}
\renewcommand{\P}{\mathbb P}
\newcommand{\Man}{{\mathbf{Man}^\infty}}
\newcommand{\an}{\arrowvert_}

\newcommand{\nok}{\setminus\{k\}}

\title{Multiple vector bundles: cores, splittings and decompositions}
\author{M.~Heuer}
\address{School of Mathematics and Statistics,
  The University of Sheffield}
\email{maheuer1@sheffield.ac.uk}
\author{M.~Jotz Lean}
\address{Mathematisches Institut, Georg-August Universit\"at G\"ottingen}
\email{madeleine.jotz-lean@mathematik.uni-goettingen.de}

\begin{document}
\maketitle

\begin{abstract}
  This paper introduces $\infty$- and $n$-fold vector bundles as
  special functors from the $\infty$- and $n$-cube categories to the
  category of smooth manifolds. We study the cores and ``n-pullbacks'' of
  $n$-fold vector bundles and we prove that any $n$-fold vector bundle
  admits a non-canonical isomorphism to a decomposed $n$-fold vector
  bundle. A colimit argument then shows that $\infty$-fold vector
  bundles admit as well non-canonical decompositions. For the
  convenience of the reader, the case of triple vector bundles is
  discussed in detail.
\end{abstract}
\tableofcontents

\section{Introduction}
Double vector bundles were introduced by Pradines \cite{Pradines77} as
a structural tool in his study of nonholonomic jets. Since then,
double vector bundles have been used e.g.~in integration problems in
Poisson geometry \cite{MaXu00, BuCaOr09, JoOr14, BuCaHo16, Jotz19},
and Pradines' symmetric double vector bundles (with inverse symmetry)
have turned out to be equivalent to graded manifolds of degree $2$
\cite{Jotz18b}. Pradines' original definition was in terms of double
vector bundle charts  \cite{Pradines77}:

\emph{
 Let $M$ be a smooth manifold and $D$ a topological space
  with a map $\Pi\colon D\to M$. A \textbf{double vector
    bundle chart} is a quintuple $c=(U,\Theta,V_1,V_2,V_0)$, where $U$
  is an open set in $M$, $V_1,V_2,V_3$ are three (finite dimensional)
  vector spaces and $\Theta\colon \Pi\inv(U)\to U\times V_1\times
  V_2\times V_0$ is a homeomorphism such that $\Pi=\pr_1\circ\Theta$.}

 \emph{ Two smooth double vector bundle charts $c$ and $c'$ are
  \textbf{smoothly compatible} if $V_i=V_i'$ for $i=0,1,2$ and 
	the ``change of chart''
  $\Theta'\circ\Theta\inv$ over $U\cap U'$ has the form
\[(x,v_1,v_2,v_0)\mapsto (x,\rho_1(x)v_1,\rho_2(x)v_2,\rho_0(x)v_0+\omega(x)(v_1,v_2))
\]
with $x\in U\cap U'$, $v_i\in V_i$, $\rho_i\in C^\infty(U\cap U',
\operatorname{Gl}(V_i))$ for $i=0,1,2$ and $\omega\in
C^\infty(U\cap U',\operatorname{Hom}(V_1\otimes V_2,V_0))$.
}
\emph{A \textbf{smooth double vector bundle atlas} $\mathfrak{A}$ on $D$
is a set of double vector bundle charts of $D$ that are
pairwise smoothly compatible and such that the family of underlying open
sets in $M$ covers $M$. 
A (smooth) double vector bundle structure on $D$ is a maximal
smooth double vector bundle atlas on $D$.}

A double vector bundle consists then of a smooth manifold $D$,
together with vector bundle structures $D\rightarrow A_1$,
$D\rightarrow A_2$, $A_1\rightarrow M$, $A_2\rightarrow M$:
\[
\begin{tikzcd}
D\ar[r,"p^D_{A_1}"] \ar[d,"p^D_{A_2}"'] & A_1 \ar[d,"q_1"]\\
A_2 \ar[r,"q_2"] & M
\end{tikzcd} 
\,,
\]
such that the structure maps (bundle projection, addition, scalar
multiplication and zero section) of $D$ over $A$ are vector bundle
morphisms over the corresponding structure maps of $B\rightarrow M$
and the other way around.
Equivalently,  the condition that each addition in
$D$ is a morphism with respect to the other is exactly
\begin{equation}\label{add_add} (d_1+_{A_1}d_2)+_{A_2}(d_3+_{A_1}d_4)=(d_1+_{A_2}d_3)+_{A_1}(d_2+_{A_2}d_4)
\end{equation}
for $d_1,d_2,d_3,d_4\in D$ with $p^D_{A_1}(d_1)=p^D_{A_1}(d_2)$,
$p^D_{A_1}(d_3)=p^D_{A_1}(d_4)$ and $p^D_{A_2}(d_1)=p^D_{A_2}(d_3)$,
$p^D_{A_2}(d_2)=p^D_{A_2}(d_4)$.  This is today's usual definition of a double
vector bundle; which has been used since \cite{Mackenzie92}.  It is
easy to see that a double vector bundle following Pradines' definition
is a double vector bundle in the ``modern'' sense \cite{Pradines77},
but the converse is more difficult to see. Pradines' double vector
bundle charts are equivalent to \emph{local linear splittings} of
today's double vector bundles. Let us be more precise.

Given three vector bundles $A$, $B$ and $C$ over $M$ with
respective vector bundle projections $q_A$, $q_B$ and $q_C$, the space
\[A\times_MB\times_MC\simeq q_A^!(B\oplus C)\simeq q_B^!(A\oplus C)\]
has two vector bundle structures, one over $A$, and one over
$B$. These two vector bundle structures are compatible in the sense of
both definitions above. Such a double vector bundle is called a
\textbf{decomposed} double vector bundle, with sides $A$ and $B$ and
with ``core'' $C$.  In particular, if $C$ is the trivial vector bundle
$M$ over $M$, we get the ``vacant'' double vector bundle $A\times_MB$
\cite{Mackenzie92}.  A (local) \textbf{linear splitting} of a double
vector bundle $(D;A,B;M)$ is an injective morphism of double vector
bundles
\[\Sigma_U\colon A\an{U} \times_{U} B\an{U}\rightarrow (q_B\circ p^D_B)^{-1}(U)\,,\]
over the identity on the sides $A\an{U}$ and $B\an{U}$, where
$U\subseteq M$ is an open subset. A (local) \textbf{decomposition} of
$(D;A,B;M)$  with core $C$ is an isomorphism of double vector bundles
\[\mathcal{S}_U\colon A\an{U}\times_{U}B\an{U}\times_{U} C\an{U}
\rightarrow (q_B\circ p^D_B)^{-1}(U),\]
which is the identity on the sides and on the core. Linear splittings
are equivalent to decompositions; and a local decomposition of $D$ as
above with the open set $U$ trivialising simultaneously $A$, $B$ and
$C$ gives a smooth double vector bundle chart of $D$, defined by
$\Theta\colon (q_B\circ p^D_B)^{-1}(U)\to U\times \mathbb R^a\times\mathbb
R^b\times \mathbb R^c$;
\[ \Theta=(\pr_U,\phi_A, \phi_B,\phi_C)\circ (\mathcal S_U)^{-1},
\]
where $a$, $b$, $c$ are the ranks of $A$, $B$, $C$, respectively and
$\phi_A\colon q_A^{-1}(U)\to U\times \mathbb R^a$ is the
trivialisation of $A$ over $U$, etc.

\medskip

Starting with the definition from \cite{Mackenzie92}, it was until
recently not known how to show the existence of local double vector
bundle charts, or equivalently of local linear splittings. In fact,
Mackenzie later added the existence of a global splitting to his
definition of a double vector bundle, and also of triple vector
bundles (see e.g.~\cite[Definition 1]{Mackenzie11}, \cite{GrMa09},
\cite{FlMa18}).  It turns out that Mackenzie's additional condition in
his definition is redundant. The existence of local splittings for the
above definition of double vector bundles has been mentioned at
several places \cite{GrMe10a,GrRo09}, but the first elementary
construction was given by Fernando del Carpio-Marek in his thesis
\cite{delCarpio-Marek15}, starting from the hypothesis that the double
projection $(p^D_A,p^D_B)\colon D\to A\times_M B$ of a double vector
bundle is a surjective submersion.

Note here that in \cite{Pradines77}, Pradines pasted local
decompositions together with a partition of unity, in order to get a
global decomposition (see in our proof of Theorem \ref{thm_dec_n_vb} below). In
other words, the existence of local decompositions is equivalent to
the existence of a global linear splitting or decomposition.

\medskip We will explain below (in Section \ref{2-case}) how to deduce
very easily from the surjectivity of the double projection
$(p^D_A,p^D_B)\colon D\to A\times_M B$ the existence of a global
splitting. This surjectivity, that is sometimes also assumed as part
of the definition of a double vector bundle (this is e.g.~done
explicitly in a former version of \cite{Mackenzie11} that can be found
on arXiv.org, and implicitly in \cite{delCarpio-Marek15}), is in fact
always ensured by Lemma \ref{lise} below (see also Remark
\ref{rem_surj_double_proj}).  Although we find a more elegant proof of
the existence of global splittings of double vector bundles than the
one in \cite{delCarpio-Marek15}, it turns out that the method there is
easier to understand and more elementary in the case of a general
$n$-fold vector bundle.  Our first goal in this project was to build
on del Carpio-Marek's method in order to construct local splittings of
triple vector bundles. It was then natural to adapt our proof to the
construction of local linear splittings of $n$-fold vector bundles;
and we found that a colimit argument yields the existence of global
linear decompositions for $\infty$-fold vector bundles as well.

Let us mention here that Eckhard Meinrenken showed us recently a
beautiful construction of global linear splittings of double vector
bundles using the normal functor, and an interesting alternative proof
to the submersive surjectivity of the double projection \cite{LiSe11},
using the commuting scalar multiplications of a double vector bundle.
\medskip

In this paper, we introduce multiple vector bundles \cite{GrMa12} as
special functors from hypercube categories to smooth manifold, such
that generating arrows are sent to vector bundle projections, and
elementary squares to double vector bundles. In particular, we define
$\infty$-fold vector bundles as such functors from the infinite
hypercube category. We study in great detail the cores of multiple
vector bundles and find on them rich structure of multiple vector
bundles as well. We define the $n$-pullback of an $n$-fold vector 
bundle and the surjective submersion onto it -- in the case of a 
double vector bundle, this is the surjectivity of
$(p^D_A,p^D_B)\colon D\to A\times_M B$ -- and most importantly we
prove by induction over $n$ that each $n$-fold vector bundle admits
local splittings and therefore a non-canonical global decomposition.

\medskip

$n$-fold vector bundles were previously defined in \cite{GrMa12},
\cite{GrRo09}. It is not difficult to see that the definitions are the
same: Gracia-Saz and Mackenzie's $n$-fold vector bundles are smooth
manifolds with $n$ ``commuting'' vector bundle structures in the sense
that all squares are double vector bundles, and Grabowski and
Rotkiewicz's are smooth manifolds with $n$ commuting scalar
multiplications.  Grabowski and Rotkiewicz sketch in \cite{GrRo09} a
proof of global splittings of their $n$-fold vector bundles. Our
construction is more precise since it explains all the multiple core
and their roles in the decomposition; and most importantly it gives
the decompositions of $\infty$-fold vector bundles with a colimit
construction. Our definition of multiple vector bundles as special
functors from cube categories to manifolds allows us to work with
$n$-fold vector bundles without giving a central role to the total
space -- an $\infty$-fold vector bundle cannot be defined as a smooth
manifold with infinitely many commuting scalar multiplications!

\subsection{Outline of the paper}
In the next section \ref{2-case} we explain for the convenience of the
reader how to prove that double vector bundles admit linear
decompositions.

In Section \ref{sec_multiple_def} we define multiple vector
bundles. We construct their pullbacks (Section \ref{sec_pullback}) and
we explain the rich structure on the different cores of multiple
vector bundles (Section \ref{cores}).

In Section \ref{sec_ex_splitting} we define linear splittings and
decompositions of $n$-fold vector bundles. We explain how the two
notions are essentially equivalent (Section \ref{splitdecnvb}) and 
we prove the existence of local splittings of a given $n$-fold vector 
bundle (Section \ref{ex_splitting}). We deduce the existence of global
decompositions of $n$-fold vector bundles and we explain how $n$-fold
vector bundles can alternatively be defined as smooth manifolds with
an atlas of compatible $n$-fold vector bundle charts (Section
\ref{sec_atlas}).

In Section \ref{sec_infty} we prove that each $\infty$-fold vector
bundle admits a linear decomposition. 
Finally in Section \ref{sec_triple} we explain for the convenience of
the reader most of our constructions and results in the case of a
triple vector bundle. In that special case, we explain the relation
between linear splittings and multiple linear sections. 

\subsection{Relation with other work}

We heard after having mostly completed this work that the content of
Theorem \ref{nPullback} for $n=3$ can be found as well in the recent
paper \cite{FlMa18}; unfortunately the proof given there
has some errors.

Some of our results on cores in Section \ref{cores} seem to be known
in \cite{GrMa12}, but they are not central in that paper so not
precisely formulated and proved. The cores of triple vector bundles
can also be found in \cite{FlMa18} and \cite{Mackenzie05b} -- our
proof of Theorem \ref{corenvb} relies on the fact that the side cores
of a triple vector bundle are double vector bundles
\cite{Mackenzie05b}.

\subsection{Acknowledgements}

We warmly thank Rohan Jotz Lean for useful comments, and Sam Morgan
for telling us about the technique used in \cite{LiSe11} for proving
that the double source map of a VB-groupoid is a surjective
submersion (used in our proof of Theorem \ref{nPullback}).

\subsection{Preparation: on linear splittings of double vector
  bundles}\label{2-case}
Let $(D,A,B,M)$ be a double vector bundle with core $C$. That is, the
space $C$ is the double kernel
$C=\{d\in D\, \mid \, p^D_A(d)=0^A_m, \quad p^D_B(d)=0^B_m\quad \text{
  for some } m\in M\}$.  It has a natural vector bundle structure over
$M$ since $+_A$ and $+_B$ of two elements of $C$ coincide by the
interchange law \eqref{add_add}, see \eqref{computation_core_add} below.

The additional axiom that the double projection
$(p^D_A,p^D_B)\colon D\to A\times_M B$ is a surjective submersion is
sometimes added to the definition.  We explain in Theorem \ref{nPullback}, see
also Remark \ref{rem_surj_double_proj}, why this additional axiom is
not needed \cite{LiSe11}.  The surjectivity of $(p^D_A,p^D_B)$ yields the exactness
of the sequence
\begin{equation}\label{core-seq}
0 \longrightarrow q_B^!C \overset{\iota_B}{\longrightarrow} D\overset{(p^D_A,p^D_B)}{\longrightarrow} q_B^!A\to 0
\end{equation}
of vector bundles over $B$.  The map $\iota_B\colon q_B^!C \to D$ is
the core inclusion over $B$; sending $(b,c)$ to $0^D_b+_Ac$. Its image
are precisely the elements of $D$ that project under $p^D_A$ to zero
elements of $A$.

A section $\xi\in\Gamma_A(D)$ is \textbf{linear} over a section
$b\in\Gamma(B)$ if the map $\xi\colon A\to D$ is a vector bundle
morphism over the base map $b\colon M\to B$. The space
$\Gamma_A^\ell(D)$ of linear sections of $D\to A$ is a
$C^\infty(M)$-module since for $\xi\in\Gamma^\ell_A(D)$ linear over
$b\in\Gamma(B)$ and for $f\in C^\infty(M)$, the section
$q_A^*f\cdot \xi$ is linear over $fb$. We get a morphism
$\pi\colon \Gamma^\ell_A(D)\to\Gamma(B)$ of $C^\infty(M)$-modules,
sending a linear section to its base section. If a linear section
$\xi\in\Gamma^\ell_A(D)$ has the zero section $0^B\in\Gamma(B)$ as its
base section, then for all $a_m\in A$,
$D\ni \xi(a_m)=0^D_{a_m}+_B\varphi(a_m)$ for some
$\varphi(a_m)\in C(m)$. The linearity of $\xi$ implies that
$\varphi\in\Gamma(A^*\otimes C)$.  We denote then $\xi$ by
$\widetilde{\varphi}$, and we get the map
$\widetilde\cdot\colon
\Gamma(\operatorname{Hom}(A,C))\to\Gamma^\ell_A(D)$
that sends $\phi$ to $\widetilde \phi\in\Gamma_A^\ell(D)$ defined by
$\widetilde\phi(a)=0^D_{a}+_B\phi(a)$ for all $a \in A$.

A splitting $s\colon q_B^!A\to D$ of \eqref{core-seq} lets us define for
every $b\in\Gamma(B)$ a section $\hat{b}$ of $D\to A$, given by 
$\hat{b}(a_m)=s(a_m,b(m))$ for all $a_m\in A$. We get then immediately
$p^D_B\circ \hat{b}=b\circ q_A\colon M\to B$ and 
\[ \hat{b}(a_m^1+a_m^2)=s(a^1_m+a^2_m,b(m))=s(a^1_m,b(m))+_Bs(a^2_m,b(m))=\hat{b}(a^1_m)+_B\hat{b}(a^2_m),
\]
i.e.
$\hat{b}\colon  A\to D$ is a vector bundle morphism over $b\colon M\to B$.
In other words, $\hat{b}$ is an element of $\Gamma^\ell_A(D)$.
Therefore, the third arrow in
\begin{equation}\label{fat_sequence_modules}
  0\longrightarrow \Gamma(\operatorname{Hom}(A,C))\overset{\widetilde\cdot}{\longrightarrow} \Gamma^\ell_A(D)\longrightarrow\Gamma(B)\longrightarrow 0
\end{equation}
is surjective and the short sequence of $C^\infty(M)$-modules is
exact. Then, since $\Gamma(\operatorname{Hom}(A,C))$ and $\Gamma(B)$
are locally free and finitely generated, $\Gamma^\ell_A(D)$ is as well
and there exists
a splitting $h\colon \Gamma(B)\to \Gamma_A^\ell(D)$ of 
\eqref{fat_sequence_modules}. Then $h$ defines a linear splitting
$\Sigma_h\colon A\times_M B\to D$, $\Sigma_h(a_m,b_m)=h(b)(a_m)$ for
any $b\in \Gamma(B)$ with $b(m)=b_m$.  
Since $h$ is $C^\infty(M)$-linear, it is easy to see that $\Sigma_h$
is well-defined, i.e.~that it does not depend on the choice of the
sections of $B$.

Hence we have proved the following theorem.
\begin{thm}
Any double vector bundle $D$ with sides $A$ and $B$ admits a linear
splitting $\Sigma\colon A\times_M B\to D$.
\end{thm}

Del Carpio-Marek proves in his thesis \cite{delCarpio-Marek15} the
existence of local splittings. His method is the following. Take a
splitting $\sigma\colon q_B! A\to D$ of the short exact sequence
\eqref{core-seq} -- here \cite{delCarpio-Marek15} seems to assume the
surjectivity of the right-hand map as an axiom in the
definition of a double vector bundle. That is, $\sigma$ is a vector
bundle morphism over the identity on $B$. Now choose $U\subseteq M$ an
open set that trivialises both $A$ and $B$ and take the induced local
frames $(a_1,\ldots,a_k)$ and $(b_1,\ldots,b_l)$ of $A$ and $B$ over
$U$.  Then each $b_m\in B\an{U}$ equals
$b_m=\sum_{i=1}^l\beta_i b_i(m)$ with
$\beta_1,\ldots,\beta_l\in\mathbb R$.  Set
$\Sigma_U\colon A\an{U}\times_UB\an{U}\to (q_B\circ p^D_B)^{-1}(U)$,
\[\Sigma_U(a_m,b_m)=\sum_{i=1}^l \beta_i\cdot_ A \sigma(a_m,b_i(m)),
\]
where the sum is taken in the fiber of $D$ over $a_m\in A$. Then
$\Sigma_U$ is a local linear splitting of $D$.

\section{Multiple vector bundles: definition and properties}\label{sec_multiple_def}
In this section we introduce multiple vector bundles and discuss some
of their properties.  The novelty of our definition is that instead of
considering an $n$-fold vector bundle as a smooth manifold with
$n$-commuting vector bundle structures, we see a multiple vector
bundle as a special functor from a cube category to smooth
manifolds. In particular, the ``total space'' of an $n$-fold vector
bundle does not play that central a role anymore, and we can even define
$\infty$-fold vector bundles, with no total space at all.

In the following, we write $\N$ for the set of positive integers:
$\N=\{1,2,\ldots\}$.  For $n\in \N$, we write $\nset$ for the set
$\{1,\ldots,n\}$.

\subsection{Multiple vector bundles}
We consider the category with objects the \emph{finite} subsets
$I\subseteq \N$ and with arrows 
\[ I\to J\quad \Leftrightarrow \quad J\subseteq I\,.
\]
We call this category the \textbf{standard $\infty$-cube category
  $\square^\N$}. It is generated as a category by the arrows
\[ I \to I\setminus\{i\} \quad \text{ for } \quad I\subseteq \N \text{ finite and } i\in I\,.
  \]
  That is, each subset $I\subseteq \N$ of cardinality $k$ is the
  source of $k$ generating arrows.

  In a similar manner, we call the \textbf{standard $n$-cube category
    $\square^n$} the category with subsets $I$ of $\nset$ as
  objects and with arrows
$I\to J\, \Leftrightarrow \, J\subseteq I$.

More generally, an \textbf{$n$-cube category} is
a category that is isomorphic to the standard $n$-cube category
  $\square^n$, while an \textbf{$\infty$-cube category} is
a category that is isomorphic to the standard $\infty$-cube category
  $\square^\N$.  \medskip

\begin{mydef}\label{def_n_fold}
  An \textbf{$\infty$-fold vector bundle}, and respectively an
  \textbf{$n$-fold vector bundle}, is a covariant functor
  $\E\colon \square^{\N}\to \Man$ -- respectively 
a covariant functor
  $\E\colon \square^{n}\to \Man$ --
 to the category of smooth manifolds,
  such that, writing $E_I$ for $\E(I)$ and $p^I_{J}:=\E(I\to J)$,
 \begin{enumerate}
    \item[(a)] for all $I\subseteq \N$ (respectively $I \subseteq \nset$) and all $i\in I$,
  $p^I_{I\setminus\{i\}}\colon E_I\to E_{I\setminus \{i\}}$ has a smooth vector bundle structure, and 
  \item[(b)] for all
  $I\subseteq \N$ (respectively $I\subseteq \nset$) and $i\neq j\in I$,
\[
\begin{tikzcd}
E_{I}\ar[rr,"p^{I}_{I\setminus\{i\}}"] \ar[d,"p^{I}_{I\setminus\{j\}}"] && E_{I\setminus \{i\}}\ar[d,"p^{I\setminus \{i\}}_{I\setminus\{i,j\}} "]\\
E_{I\setminus \{j\}} \ar[rr,"p^{I\setminus \{j\}}_{I\setminus\{i,j\}}  "] & &E_{I\setminus \{i,j\}} 
\end{tikzcd} 
\]
is a double vector bundle.
\end{enumerate}
\end{mydef}
For better readability we will often write for the vector bundle
projections $p^I_i:=p^I_{I\setminus\{i\}}$ and in the case of an
$n$-fold vector bundle also $p_i:=p^\nset_{\nset\setminus\{i\}}$. The
smooth manifold $E_\emptyset=:M$ will be called the \emph{absolute 
base} of $\E$. If $\E$ is an $n$-fold vector bundle, the smooth
manifold $\E(\nset)=:E$ is called its \emph{total space}. Given a
finite subset $I\subseteq \N$ and $i\in I$, we write
$+_{I\setminus\{i\}}$ for the addition and $\cdot_{I\setminus\{i\}}$
for the scalar multiplication of the vector bundle 
$E_{I}\to E_{I\setminus\{i\}}$. This notation is omissive since it 
only specifies the base space of the vector bundle in the fibers of 
which the addition or scalar multiplication is taken. However, it is 
always clear from the summands or factors which fiber space is considered.

We will generally say \emph{multiple vector bundle} for an $n$-fold or
$\infty$-fold vector bundle, when the dimension of the underlying cube
diagram does not need to be specified. Our definition of $n$-fold vector 
bundles is different but equivalent notation to the definition in \cite{GrMa12}. 

\begin{rmk}
  There is a canonical functor $\pi^n_k\colon\square^n\to\square^k$
  for $k\leq n$ defined by $\pi^n_k(I)=I\cap \underline{k}$ and
  $\pi^n_k(I\to J)=(I\cap \underline{k})\to (J\cap \underline{k})$.
  The canonical functor $\pi^\N_n\colon \square^\N\to\square^n$ is
  defined in the same manner by $\pi^\N_n(I)=I\cap \nset$.
  Furthermore there are inclusion functors of full subcategories
  $\iota^n_k\colon \square^k\to\square^n$ and
  $\iota^\N_n\colon \square^n\to\square^\N$.

Given  a $k$-fold vector bundle $\E:\square^k\to \Man$, the composition
$\E\circ\pi^n_k$ is an $n$-fold vector bundle whereas the composition
$\E\circ\pi^\N_k$ is an $\infty$-fold vector bundle. 

In this light, a standard $n$-fold vector bundle $\E$ can be viewed as
a special case of a standard $\infty$-fold vector bundle
$\E\colon \square^\N\to\Man$ such that additionally
$\E=\E\circ\iota^\N_n\circ\pi^\N_n$:
	\[
	\begin{tikzcd}
	\square^\N \ar[d,"\pi^\N_n"'] \ar[r,"\E"] & \Man \\
	\square^n \ar[r,"\iota^\N_n"'] & \square^\N\ar[u,"\E"]
	\end{tikzcd}
	\,.
	\]
	In other words $\E(I)=\E(I\cap\nset)$ for all
	$I\subseteq\N$ and $\E$ is completely determined by its values on
	all the subsets of $\nset$ already. 
\end{rmk}

We will also more generally call an \textbf{$n$-fold vector bundle} a
functor $\E\colon\lozenge^n\to\Man$, where $\lozenge^n$ is an $n$-cube
category with isomorphism $\mathbf{i}\colon \square^n\to \lozenge^n$,
such that $\E\circ \mathbf i$ is a standard $n$-fold vector bundle.
Similarly, an \textbf{$\infty$-fold vector bundle} is a functor
$\E\colon\lozenge^\N\to \Man$, where $\lozenge^\N$ is an $\infty$-cube
category with isomorphism $\mathbf i\colon\square^\N\to\lozenge^\N$,
such that $\E\circ \mathbf i$ is a standard $\infty$-fold vector
bundle. We need this generality of the definition for the study of the
cores of a multiple vector bundle.

The following proposition is straightforward and its proof is left to the reader. 
  \begin{prop}\label{sub-k-fold}
    Let $\E\colon \square^{\mathbb N}\to \Man$ be a multiple vector
    bundle.
    \begin{enumerate}
    \item[(a)]
			For each pair of subsets $J\subseteq I\subseteq\N$ with $J$ finite, 
			the finite sets $K\subset \N$ such that $J\subseteq K\subseteq I$ 
			form a full subcategory $\lozenge^{I,J}$ of $\square^\N$, 
			which is itself a $(\# I-\# J)$-cube category and the restriction 
			of $\E$ to $\lozenge^{I,J}$ is a 
      $(\# I-\# J)$-fold vector bundle with total space $E_I$ (if $I$ is finite)
      and absolute base $E_J$, denoted by $\E^{I,J}$. We call this the
			$(I,J)$-face of $\E$. 
    \item[(b)] In particular, if $I=\emptyset$ we obtain a $(\# I)$-fold vector bundle 
			$\E^{I,\emptyset}$ with total space $E_I$ and absolute base $M$. 
			We call $\E^{I,\emptyset}$ the $I$-face of $\E$. 
\end{enumerate}
\end{prop}

\medskip 

Given an $\infty$-fold vector bundle
$\E\colon \square^{\N}\to \Man$ and an open subset $U\subseteq M$, we
define the \emph{restriction of $\mathbb E$ to $U$} to be the
$\infty$-fold vector bundle
$\mathbb E\an{U}\colon \square^{\N}\to \Man$,
$\mathbb E\an{U}(I)=\left(p_\emptyset^{I}\right)\inv(U)$ and
$\mathbb E\an{U}(I\to J)=\mathbb E(I\to
J)\an{(p_\emptyset^{I})\inv(U)}\colon
\left(p_\emptyset^{I}\right)\inv(U)\to
\left(p_\emptyset^{J}\right)\inv(U)$.  The absolute base of
$\mathbb E\an{U}$ is $U$.  In the same manner, if
$\E\colon \lozenge^{n}\to \Man$ is an $n$-fold vector bundle, and $U$
an open subset of $M$, then its restriction $\E\an{U}$ to $U$ is an
$n$-fold vector bundle with total space $(p^{\nset}_\emptyset)\inv(U)$
and with absolute base $U$.

\medskip

Now recall that a double vector bundle morphism $(\Psi; \psi_A,
  \psi_B; \psi)$ from $(D_1,A_1,B_1,M_1)$ to $(D_2,A_2,B_2,M_2)$ is a commutative cube
{\small\begin{equation*}
 \begin{tikzcd}
   D_1\ar[rrr, "\Psi"]\ar[rd, 
   ]\ar[dd
   ]& && D_2\ar[rd
   ]\ar[dd]
      &\\
& B_1\ar[dd]\ar[rrr, "\psi_B"]&&&B_2\ar[dd
]\\
A_1\ar[rd
]\ar[rrr, "\psi_A"]&&&A_2\ar[rd
]&\\
&M_1\ar[rrr, "\psi"]&&&M_2
\end{tikzcd}
\end{equation*}}
all thee faces of which are vector bundle morphisms. 
Similarly we define morphisms of multiple vector bundles.

\begin{mydef}
Let $\E\colon \lozenge^{\N}_1\to \Man$ and
$\F\colon \lozenge^{\N}_2\to \Man$ be two multiple
vector bundles. A \textbf{morphism of multiple vector bundles} 
from $\E$ to $\F$ is a natural transformation 
$\tau\colon \E\circ \mathbf{i}_1\to \F\circ\mathbf{i}_2$ 
such that for all objects $I$ of $\square^{\N}$ and for all $i\in I$, 
the commutative diagram
\[
\begin{tikzcd}
E_{\mathbf{i_1}(I)}\ar[r,"\tau(I)"] \ar[d,"p^{\mathbf{i_1}(I)}_{\mathbf{i_1}  (I\setminus\{i\})}"] & F_{\mathbf{i_2} (I)}\ar[d,"p^{\mathbf{i_2} (I)}_{\mathbf{i_2} (I\setminus\{i\})} "]\\
E_{\mathbf{i_1}(I\setminus \{i\})} \ar[r,"\tau(I\setminus\{i\})  "] & F_{\mathbf{i_2}(I\setminus \{i\})} 
\end{tikzcd} 
\]
is a homomorphism of vector bundles. 

Given two $n$-fold vector bundles 
$\E\colon \lozenge_1^n\to \Man$ and
$\F\colon \lozenge_2^n\to \Man$, a 
\textbf{morphism of $n$-fold vector bundles} from $\E$ to $\F$ is a
natural transformation $\tau\colon \E\circ\mathbf{i_1}\to \F\circ\mathbf{i_2}$ 
such that the diagram above is a vector bundle homomorphism for all
$I\subseteq \nset$ and $i\in I$. The morphism $\tau$ is surjective 
(resp. injective) if each of its components $\tau(I)$, $I\subseteq \nset$ 
is surjective (resp. injective).
\end{mydef}

\subsection{Prototypes}\label{dec_n_fold}
In this section, we describe a few standard examples of multiple
vector bundles, that will be relevant in the formulation of our main
theorem.

\subsubsection{\textbf{Decomposed multiple and $n$-fold vector bundles}}
Consider a smooth manifold $M$ and a collection of vector bundles
  $\mathcal A=(q_J\colon A_J\to M)_{J\subseteq \N,\, \# J<\infty}$, with $A_\emptyset =M$.  We define a
  functor
  $\E^{\mathcal A}\colon \square^{\N}\to \Man$
  as follows.  Each finite subset $I\subseteq \N$ is sent to
  $E_I:=\prod^M_{J\subseteq I} A_J$, the fibered product of vector
  bundles over $M$. 
	
  For $I\subseteq \N$ with $1\leq\# I<\infty$ and for $k\in I$, the
  arrow $I\to I\setminus\{k\}$ is sent to the canonical vector bundle
  projection
  \[p^I_k\colon \prod^M_{J\subseteq I} A_J\to \prod^M_{J\subseteq
      I\setminus\{k\}} A_J.\] In particular, the arrow
  $\{i\}\to\emptyset$ for $i\in\N$ is sent to the vector bundle
  projection
  $p^{\{i\}}_\emptyset=q_{\{i\}}\colon E_{\{i\}}=A_{\{i\}}\to
  E_\emptyset=M$.  A multiple vector bundle
  $\E^{\mathcal A}\colon \square^{\N}\to \Man$ constructed in this
  manner is called a \textbf{decomposed multiple vector
    bundle}. 
  A decomposed $n$-fold vector bundle
  $\E^{\mathcal A}\colon \square^{n}\to \Man$ is
  defined accordingly. In that case we will write $E^\A:=\E^\A(\nset)$ 
  for the total space.
	Decomposed $n$-fold vector bundles are also defined in \cite{GrMa12}.

  \begin{example}\label{decTVB}
  A $3$-fold vector bundle is also called a triple vector bundle.
  A \emph{trivial or decomposed triple vector bundle} is given by 
  \[
	E_{\{1,2,3\}}=A_{\{1\}}\times_MA_{\{2\}}\times_MA_{\{3\}}\times_MA_{\{1,2\}}
	\times_MA_{\{1,3\}}\times_MA_{\{2,3\}}\times_MA_{\{1,2,3\}},
  \]
  with decomposed sides
\begin{equation*}
		\begin{split}
		E_{\{1,2\}}&=A_{\{1\}}\times_MA_{\{2\}}\times_MA_{\{1,2\}}\,,\quad 
		E_{\{1,3\}}=A_{\{1\}}\times_MA_{\{3\}}\times_MA_{\{1,3\}}\,,\\ 
		E_{\{2,3\}}&=A_{\{2\}}\times_MA_{\{3\}}\times_MA_{\{2,3\}}\,,
		\end{split}
\end{equation*}
where $A_I$, $I\subseteq \nset$ are all vector bundles over
$M$, the projections are the appropriate projections to the factors
and the additions are defined in an obvious manner in the fibers.
\end{example}
  
\subsubsection{\textbf{Vacant multiple and $n$-fold vector bundles}}
As a special case of this, if 
$\overline{\mathcal A}=(q_i\colon A_i\to M)_{i\in\N}$ is a collection
of vector bundles over $M$, we construct the multiple vector bundle
$\E^{\overline{\A}}\colon \square^{\N}\to
\Man$ as follows:
  \[ I\mapsto \prod^M_{i\in I} A_{i}, \qquad (I\to I\setminus\{k\})
  \mapsto \left(p^I_k\colon \prod^M_{i\in I} A_{i}\to
    \prod^M_{i\in I\setminus\{k\}} A_{i}\right).
\]Such a multiple vector bundle is called a \textbf{vacant decomposed}
multiple vector bundle. We will see later that all \emph{cores} of
these multiple vector bundles are trivial.
    
  Given a collection of vector bundles
  $\mathcal A=(q_J\colon A_J\to M)_{J\subseteq \N,\,\#J<\infty}$, with
  $A_\emptyset =M$, we can define
  $\overline{\mathcal A}=(q_i\colon A_i\to M)_{i\in\N}$ by
  $A_i=A_{\{i\}}$. We get then a monomorphism of multiple vector
  bundles
    \begin{equation}\label{emb_spl_dec}
\iota\colon \E^{\overline{\A}}\to \E^\A
    \end{equation}
    defined by
    $\iota(I)\colon \prod^M_{i\in I} A_{\{i\}}\to \prod^M_{J\subseteq
      I} A_J$,
    $\iota(I)((v_i)_{i\in I})=(w_J)_{J\subseteq I}$, $w_{\{i\}}=v_i$
    for $i\in I$, $w_\emptyset=v_\emptyset:=m\in M$ and
    $w_J=0^{A_J}_m$ for $\# J\geq 2$. In particular,
    $\iota(\{i\})=\id_{A_{\{i\}}}$ for all $i\in \N$.
    
    In the case of an $n$-fold vector bundle we write 
    $\overline{E}:=\overline{\E}(\nset)$ for the total space. 

    \subsubsection{\textbf{``Diagonal'' decomposed and vacant $k$-fold vector bundles}}
		\label{diagdec}
    More generally, consider a collection
    $\mathcal A=(q_I\colon A_I\to M)_{I\subseteq \nset}$ of
    vector bundles, with $A_\emptyset =M$, and a partition $\rho=\{I_1,\ldots,I_k\}$ of
    $\nset$ with $I_j\neq \emptyset$, for $j=1,\ldots,k$. Then we can define a
    $k$-cube category $\lozenge^{\rho}$ with objects the subsets
    $\nu\subseteq \rho$ and with morphisms
    $\nu_1\to \nu_2 \Leftrightarrow \nu_2\subseteq \nu_1$. 
		We will write $[\nu]:=\cup_{K\in\nu}K$ for $\nu\subseteq\rho$.
    Now we define a vacant $k$-fold vector bundle
    $\overline{\E^{\mathcal A}_\rho}\colon \lozenge^{\rho}\to\Man$ by
    \[ \nu\mapsto \prod^M_{K\in \nu} A_{K}, \qquad
      (\nu\to\nu\setminus\{I\})\mapsto \left(p^\nu_{\nu\setminus\{I\}}\colon
        \prod^M_{K\in \nu} A_{K}\to \prod^M_{K\in \nu\setminus\{I\}}
        A_{K}\right).
\]
In a similar manner, we define a decomposed 
$k$-fold vector bundle
    $\E^{\mathcal A}_\rho\colon \lozenge^{\rho}\to\Man$ by
    \[ \nu\mapsto \prod^M_{
    \nu'\subseteq\nu} A_{[\nu']}, \qquad
(\nu\to\nu\setminus\{I\})\mapsto \left(\prod^M_{\nu'\subseteq\nu} A_{[\nu']}\to 
\prod^M_{\nu'\subseteq \nu\setminus\{I\}} A_{[\nu']}
\right),
\]
where the map on the right-hand side is the canonical projection.  We
get as before an obvious monomorphism of $k$-fold vector bundles
$\iota^\rho\colon \overline{\E^{\mathcal A}_\rho}\to
\E^{\mathcal A}_\rho$.
For each $\nu\subseteq \rho$ we have furthermore the obvious canonical injections
\[ \eta^\rho(\nu)\colon \E^{\mathcal A}_\rho(\nu)=\prod^M_{\nu'\subseteq \nu}A_{[\nu']}\hookrightarrow
  \E^{\mathcal A}([\nu])=\prod^M_{J\subseteq [\nu]}A_J\,.\]

\subsubsection{\textbf{The tangent prolongation of an $n$-fold vector bundle}}
Given an $n$-fold vector bundle $\E\colon \square^{n}\to \Man$ 
we define an $(n+1)$-fold vector bundle 
$T\E\colon \square^{n+1}\to \Man$, the \textbf{tangent prolongation of} 
$\E$, as follows. Given $I\subseteq\nset$, we set $T\E(I):=E_I$ 
and $T\E(I\cup\{n+1\}):=TE_I$. Furthermore, for $i\in I\subseteq\nset$ 
we set 
\begin{align*}
&T\E(I\to I\setminus\{i\}):=p^I_i
\colon E_I\to E_{I\setminus\{i\}}\,,\\
&T\E(I\cup\{n+1\}\to (I\cup\{n+1\})\setminus\{i\}):=T(p^I_i)
\colon TE_I\to TE_{I\setminus\{i\}}\,,\\ 
&T\E(I\cup\{n+1\}\to I):=p_{E_I}
\colon TE_I\to E_I\,,
\end{align*}
where the last map is the canonical projection. 

\subsubsection{\textbf{Multiple homomorphism vector bundles}}
Given two $n$-fold vector bundles $\E$ and $\F$ with the same absolute 
base $\E(\emptyset)=\F(\emptyset)=M$ we construct an
$n$-fold vector bundle $\Hom_n(\E,\F)$, which is the $n$-fold analogon
of the bundle $\Hom(E,F)$ for ordinary vector bundles $E$ and $F$ over 
$M$.  

For $m\in M$ the restrictions $\E|_m$ and $\F|_m$ define $n$-fold
vector bundles over a single point as absolute base. With this we can
define $\Hom_n(\E,\F)$ to be
\[
\Hom_n(\E,\F):=\bigl\{\Phi_m\colon \E|_m\to \F|_m\mid m\in M, \, 
\Phi_m \text{ morphism of }n\text{-fold vector bundles} \bigr\}\,.
\]
This space is equipped with an obvious projection to $M$.
Since $n$-fold vector bundle morphisms have underlying $(n-1)$-fold
vector bundle morphisms between the faces there are additionally
projections
$\Hom_n(\E,\F)\to
\Hom_{n-1}(\E^{\nset\nok,\emptyset},\F^{\nset\nok,\emptyset})$ for all
$k\in\nset$. Each of these projections carries a vector bundle
structure, with the sum of two morphisms $\Phi_m$ and $\Psi_m$
projecting to the same base
$\phi\colon \E^{\nset\nok}|_m\to \F^{\nset\nok}|_m$ defined as
$(\Phi_m +_{\nset\nok}\Psi_m)(e):=\Phi_m(e) +_{\nset\nok} \Psi_m(e)$.  These
vector bundle structures define an $n$-fold vector bundle
$\Hom(\E,\F)$ with total space $\Hom_n(\E,\F)$ and absolute base $M$,
by setting
$\Hom(\E,\F)(I):=\Hom_{\# I}(\E^{I,\emptyset},\F^{I,\emptyset})$.

Every morphism of $n$-fold vector bundles $\E\rightarrow\F$ over the
identity on $M$ corresponds to a smooth map $M\to \Hom_n(\E,\F)$ which
is a section of the projection to $M$.  

In particular, let $F\to M$ be an ordinary vector bundle and 
consider the $n$-fold vector bundle $\F$ defined by 
$\F(\nset)=F$ and $\F(I)=M$ for all $I\subsetneq\nset$. 
Then we write $\Mor_n(\E, F)$ for the space of $n$-fold vector
bundle morphisms from $\E$ to $\F$ over $\id_M$.
\begin{lemma}\label{module_morphisms}
  Let $\E$ be an $n$-fold vector bundle over $M$ and $F$ be a vector 
	bundle over $M$. Then the space
  $\Mor_n(\E, F)$ is a
  $C^\infty(M)$-module.
  \end{lemma}

  \begin{proof}
    An element $\tau$ of $\Mor_n(\E, F)$ necessarily satisfies
    $\tau(I)\colon E(I)\to M$, $\tau(I)(e)=p^I_\emptyset(e)$ for all
    $e\in \E(I)$, $I\subsetneq \nset$.  Take $f_1, f_2\in C^\infty(M)$
    and $\tau_1,\tau_2\in \Mor_n(\E, F)$. Then
    $(f_1\cdot \tau_1+f_2\cdot \tau_2)\colon \E\to F$ is
    defined by
    $(f_1\cdot\tau_1+f_2\cdot \tau_2)(I)(e)=p^I_\emptyset(e)$ for all
    $e\in \E(I)$, $I\subsetneq \nset$ and
    $(f_1\cdot\tau_1+f_2\cdot
    \tau_2)(\nset)(e)=f_1(p^I_\emptyset(e))\cdot\tau_1(e)+f_2(p^I_\emptyset(e))\cdot\tau_2(e)$
    for $e\in \E(\nset)$.

    By construction, $(f_1\cdot\tau_1+f_2\cdot \tau_2)(\nset)$ is smooth and
    \[
\begin{tikzcd}
\E(\nset)\ar[rr,"(f_1\tau_1+f_2\tau_2)(\nset)"] \ar[d,"p^{\nset}_{\nset\setminus\{i\}}"] && F\ar[d,"q_{F} "]\\
\E(\nset\setminus \{i\}) \ar[rr,"\tau(\nset\setminus\{i\})"] && M 
\end{tikzcd} 
\]
is a morphism of vector bundles for all $i\in \nset$.  For
$I\subsetneq \nset$ and $i\in I$, the map
$(f_1\cdot\tau_1+f_2\cdot \tau_2)(I)\colon \E(I)\to M$ is obviously a
vector bundle morphism over
$\tau(I\setminus\{i\})\colon \E(I\setminus\{i\})\to M$.
    \end{proof}

\subsection{The $n$-pullback of an $n$-fold vector bundle}\label{sec_pullback}

Let $\E$ be an $n$-fold vector bundle.
We define the $n$\textbf{-pullback of }$\E$ to be the set
\[P=\left\{(e_1,\ldots,e_n) \left| e_i\in E_{\nset\setminus\{i\}} 
		\text{ and } p_j^{\underline{n}\setminus\{i\}}(e_i)= p_i^{\nset\setminus\{j\}}(e_j)
			\text{ for } i,j\in \nset\right.\right\}\,.
  \]

We prove the following theorem, which is central in our proof of the existence of a linear splitting.
  \begin{thm}\label{nPullback}
Let $\E\colon \square^n\to \Man$ be an $n$-fold vector bundle. Then 
\begin{enumerate}
\item[(a)] $P$ defined as
above is a smooth embedded submanifold 
of the product $E_{\nset\setminus\{1\}}\times \ldots \times E_{\nset\setminus\{n\}}$.
\item[(b)] The functor $\P$
    defined by $\P(\nset)=P$, $\P(S)=E_S$ for all $S\subsetneq \nset$
    and the vector bundle projections
    $p^S_{i}\colon E_S\to E_{S\setminus\{i\}}$ for all
    $S\subsetneq \nset$ and $i\in S$ and
    $p'_{i}\colon P\to E_{\nset\setminus\{i\}}$,
    $(e_1,\ldots,e_n)\mapsto e_i$ is an $n$-fold vector bundle.
  \item[(c)] The map $\pi(\nset)\colon E\to P$ given by $\pi(\nset)\colon e\mapsto (p_1(e),\ldots, p_n(e))$,
    defines together with $\pi(J)=\id_{E_J}$ for $J\subsetneq \nset$,
    a surjective $n$-fold vector bundle morphism $\pi\colon \E \to \P$.
\end{enumerate}
\end{thm}
Note that for each $i\in\nset$, the top map $\pi(\nset)\colon E\to P$
of $\pi$ is necessarily  a vector
bundle morphism over the identity on $E_{\nset\setminus\{i\}}$.
For the proof of this theorem, we need the following lemmas.
\begin{lemma}\label{lemma_embedding}
  Let $f\colon M\to N$ be a smooth surjective submersion, and let
  $q_E\colon E\to N$ be a smooth vector bundle. Then the inclusion
  $f^!E\hookrightarrow E\times M$
  is a
  smooth embedding.
  \end{lemma}
This lemma is standard and its proof is left as an exercise. The next statement is obvious.

\begin{lemma}\label{surj_surj}
  Let $A\to M$ and $B\to N$ be two smooth vector bundles, and let
  $\phi\colon A\to B$ be a homomorphism of vector bundles over a
  surjective submersion $f\colon M\to N$. Assume that $\phi$ is
  surjective in each fiber. Then the pullback
  homomorphism $f^!\phi\colon A\to f^!B$, $a_m\mapsto (\phi(a_m),m)$
  over the identity on $M$ is surjective in each fiber.
\end{lemma}

The following lemma is central in our proof, its technique is inspired
by a similar one in \cite{LiSe11}.
\begin{lemma}\label{lise}
Let $A\to M$ and $B\to N$ be two smooth vector bundles, and let
  $\phi\colon A\to B$ be a homomorphism of vector bundles over a
  smooth map $f\colon M\to N$. Then $\phi$ is a surjective submersion 
if and only if $\phi$ is surjective in each fiber and $f$ is a
surjective submersion.
\end{lemma}

\begin{proof}
  Choose $a_m\in A$. Then it is easy to see in local coordinates that
  the tangent space $T_{a_m}A$ splits as
  $T_{a_m}A\simeq T_mM\oplus A(m)$, and the tangent space
  $T_{\phi(a_m)}B$ splits as $T_{f(m)}N\oplus B(f(m))$. In those
  splittings, the map $T_{a_m}\phi\colon T_{a_m}A\to T_{\phi(a_m)}B$
  reads 
\[ T_{a_m}\phi=T_mf\oplus \phi\arrowvert_{A(m)}\colon T_mM\oplus
A(m)\to  T_{f(m)}N\oplus B(f(m)).
\]
Therefore, $T_{a_m}\phi$ is surjective if and only if $T_mf\colon
T_mM\to T_{f(m)}N$ is surjective and\linebreak $\phi\arrowvert_{A(m)}\colon
A(m)\to B(f(m))$ is surjective. Since the surjectivity of $\phi$
implies the surjectivity of $f$, the proof can easily be completed.
\end{proof}

\begin{rmk}\label{rem_surj_double_proj}
  Take $D$ a double vector bundle with sides $A$ and $B$. Then
  $q_B\colon B\to M$ is a surjective submersion since it it a vector
  bundle projection, and $p^D_A\colon D \to A$ is a surjective
  submersion for the same reason. Hence Lemma \ref{lise} implies that
  $p^D_A$ is surjective in each fiber. Now if $A\times_M B$ is
  identified with $q_B^!A$, then
  $(p^D_A,p^D_B)\colon D\to A\times_M B$ coincides with the pullback
  morphism $q_B^!p^D_A\colon D\to q_B^!A$ as morphism of vector
  bundles over $B$. By Lemma \ref{surj_surj}, it is hence surjective
  in each fiber, and so $(p^D_A,p^D_B)\colon D\to A\times_M B$ is
  surjective.  This shows Theorem \ref{nPullback} in the case $n=2$
  since then $A\times_M B$ is an embedded submanifold of $A\times B$,
  it is the total space of a double vector bundle with sides $A$ and
  $B$ and with trivial core, and the projection
  $\pi(\{1,2\})\colon D\to A\times_M B$ is equal to $(p^D_A,p^D_B)$.
  This reasoning is due to \cite{LiSe11}, and the proof of Theorem
  \ref{nPullback} is just a generalisation of it to the case of an
  arbitrary $n$, with a central role of Lemma \ref{lise} and
  of Lemma \ref{surj_surj}.
\end{rmk}

\begin{lemma}\label{emb_prod}
Let $q_A\colon A\to M$ be a smooth vector bundle, and let $B\subseteq A$ and
$N\subseteq M$ be embedded submanifolds with $q_A(B)=N$ and such that
for each $n\in N$, $B(n)\subseteq A(n)$ is a vector subspace. Then
$B\to N$ has a unique smooth vector bundle structure, such that the smooth
embeddings build a vector bundle homomorphism into $A\to M$.
\end{lemma}
This last lemma is standard as well. We leave its proof to the reader.

  \begin{proof}[Proof of Theorem \ref{nPullback}]
		We prove this by induction over $n$.
The case of $n=1$ is trivially 
		satisfied since in that case $\E$ is an ordinary vector bundle
                $E=E_{\{1\}}\to E_\emptyset =M$ and so $P=M$. 
		Let us now take $n\in\N$ with $n\geq 2$ and assume that all three claims are true for any 
		$(n-1)$-fold vector bundle $\E$.
		
Recall from Proposition \ref{sub-k-fold} that $\E^{\nset,\{k\}}$ is an $(n-1)$-fold vector bundle. The corresponding
$(n-1)$-pullback is
\[P^\upp_k:=\left\{(e_1,\ldots,\widehat{k},\ldots,e_n)\mid e_i\in
    E_{\nset\setminus \{i\}} \colon p_j^{\nset\setminus \{i\}}(e_i)=
    p_i^{\nset\setminus\{j\}}(e_j)\text{ for } i,j\in\nset\setminus
    \{k\}\right\}\,.\] By the induction hypothesis (b), this is the total space
of an $(n-1)$-fold vector bundle $\P^\upp_k$ with underlying nodes $E_J$ for
$k\in J\subsetneq \nset$. The absolute base of this $(n-1)$-fold
vector bundle is $E_{\{k\}}$, and by (c) we have a smooth morphism
$\pi^\upp_k\colon \E^{\nset,\{k\}}\to \P^\upp_k$ of $(n-1)$-fold
vector bundles
that is surjective.
In a similar manner, $\E^{\nset\setminus\{k\},\emptyset}$ is an
$(n-1)$-fold vector bundle.  The corresponding $(n-1)$-pullback is
\[P^{\low}_k:=\left\{(b_1,\ldots,\widehat{k},\ldots,b_n)\mid b_i\in
    E_{\nset\setminus \{k,i\}} \colon p_j^{\nset\setminus
      \{k,i\}}(b_i)= p_i^{\nset\setminus\{k,j\}}(b_j)\text{ for }
    i,j\in\nset\setminus\{k\}\right\}\,.\] Again by the induction
  hypothesis (b)
this is the total space of an $(n-1)$-fold vector bundle $\P^{\low}_k$
with
underlying nodes $E_J$ for $J\subsetneq \nset\setminus\{k\}$.
By (c) we have a smooth surjective morphism
$\pi^\low_k\colon \E^{\nset\setminus\{k\},\emptyset}\to \P^\low_k$ of $(n-1)$-fold
vector bundles.

By the induction hypothesis (a), $P^\upp_k$ and $P^\low_k$ are
embedded submanifolds of
$\prod_{\substack{i=1\\i\neq k}}^n E_{\nset\setminus\{i\}}$ and
$\prod_{\substack{i=1\\i\neq k}}^n E_{\nset\setminus\{i,k\}}$,
respectively. Since for each $i\neq k$ in $\nset$, we have the smooth
vector bundle
$ p^{\nset\setminus\{i\}}_k\colon E_{\nset\setminus\{i\}}\to
E_{\nset\setminus\{i,k\}}$,
the product
$\prod_{\substack{i=1\\i\neq k}}^n E_{\nset\setminus\{i\}}$ has a
smooth vector bundle structure over
$\prod_{\substack{i=1\\i\neq k}}^n E_{\nset\setminus\{i,k\}}$, the
projection of which we denote by $q_k$. Using the surjectivity of
$\pi^\low_k(\nset\setminus\{k\})\colon E_{\nset\setminus\{k\}}\to
P^\low_k$,
the surjectivity of $p_k\colon E\to E_{\nset\setminus\{k\}}$, as well
as the identities
$p^{\nset\setminus\{k\}}_i\circ p_k=p^{\nset\setminus\{i\}}_k\circ
p_i$
for $i\neq k$, we find easily that $q_k(P^\upp_k)=P^\low_k$. Further,
$P^\upp_k$ is clearly closed under the addition of
$\prod_{\substack{i=1\\i\neq k}}^n
E_{\nset\setminus\{i\}}\to\prod_{\substack{i=1\\i\neq k}}^n
E_{\nset\setminus\{i,k\}}$.
Lemma \ref{emb_prod} yields then that $q_k\colon P^\upp_k\to P^\low_k$ is a
smooth vector bundle.

Next let us set for simplicity
$\delta_k:=\pi^\low_k(\nset\setminus\{k\})\colon
E_{\nset\setminus\{k\}}\to P^{\low}_k$. Recall that it is defined by
\[\delta_k\colon
e_k\mapsto\left(p^{\nset\setminus\{k\}}_1(e_k),\ldots, \hat k,\ldots,
  p^{\nset\setminus\{k\}}_n(e_k)\right)\,.\]
Since $n\geq 2$ we can choose $i\in\nset\setminus\{k\}$. Then $\delta_k\colon
E_{\nset\setminus\{k\}}\to P^{\low}_k$
is a surjective smooth vector bundle homomorphism over the identity on
$E_{\nset\setminus \{i,k\}}$. By Lemma \ref{lise}, it is a
surjective submersion.
We consider the pullback vector bundles $(\delta_k)^!P^\upp_k$ over
$E_{\nset\setminus\{k\}}$, for each $k\in\nset$.  As a set, each
$(\delta_k)^!P^\upp_k$can easily be identified with $P$.  

Denote by $\varphi_k$ the inclusion of $P^\upp_k$ in
$E_{\nset\setminus\{1\}}\times\ldots\hat k\ldots\times
E_{\nset\setminus\{n\}}$. Then $P$ is embedded into
$E_{\nset\setminus\{1\}}\times\ldots\times E_{\nset\setminus\{n\}}$
via the composition
\[\begin{tikzcd}
  P\ar[r,hook] & 
  P^\upp_k\times E_{\nset\setminus\{k\}}\ar[rr,hook, "\varphi_k\times
    \id_{E_{\nset\setminus\{k\}}}"] & &
  (E_{\nset\setminus\{1\}}\times\ldots\hat k\ldots\times
  E_{\nset\setminus\{n\}})\times E_{\nset\setminus\{k\}}
\end{tikzcd}\,,\] where the map on the left is the embedding as in
Lemma \ref{lemma_embedding}.  It is easy to see that up to the obvious
reordering of the factors on the right, the embeddings obtained for
$k=1,\ldots,n$ are the same map. Therefore, all the obtained smooth
structures on $P$ are compatible and so $P$ is a smooth manifold and
all its projections are smooth. In particular, we have proved (a).

The compatibility of the vector bundle structures of $P$ over
$E_{\nset\setminus\{i\}}$ and $E_{\nset\setminus\{j\}}$ for $i\neq j$
follows from the compatibility of the structures in
$\E^{\nset\setminus\{k\},\emptyset}$. More precisely for
$i,j\in\nset$, the interchange law in the double vector bundle
$(P,E_{\nset\setminus\{i\}},E_{\nset\setminus\{j\}},E_{\nset\setminus\{i,j\}})$
follows from the interchange laws in the double vector bundles
$(E_{\nset\setminus\{k\}},E_{\nset\setminus\{k,i\}},E_{\nset\setminus\{k,j\}},E_{\nset\setminus\{k,i,j\}})$
for all $k\in\nset\setminus\{i,j\}$. We let the reader check this as
an exercise.  Hence we can define $\P\colon \square^n\to \Man$ and we
obtain an
$n$-fold vector bundle.

\medskip
 
For each $k=1,\ldots,n$, $\pi_k^\upp(\nset)\colon E\to P^\upp_k$ is a
vector bundle morphism over
$\delta_k\colon E_{\nset\setminus\{k\}}\to P^\low_k$.  The pullback of
$\pi_k^\upp(\nset)$ via the map $\delta_k$ is hence a vector bundle
morphism $E\to (\delta_k)^!P^\upp_k$ over the identity on
$E_{\nset\setminus\{k\}}$, and it is easy to see that it coincides --
via the identification of $P$ with $(\delta_k)^!P^\upp_k$ -- with the
$n$-fold projection $\pi(\nset)$ from $E$ to $P$.  Hence
$\pi\colon\E\to\P$ is an $n$-fold vector bundle morphism.

As before choose $i\in\nset\setminus\{k\}$. Since
$\pi_k^\upp(\nset)\colon E\to P^\upp_k$ is a surjective vector bundle morphism
over the identity on $E_{\nset\setminus\{i\}} $, it is a
surjective submersion by Lemma \ref{lise}. But since
$\delta_k\colon E_{\nset\setminus\{k\}}\to P^\low_k$ is a surjective
submersion and $\pi_k^\upp(\nset)$ is a vector bundle morphism over
$\delta_k$, by Lemma \ref{lise} it must be surjective in each fiber of
$p_k\colon E\to E_{\nset\setminus\{k\}}$. By Lemma \ref{surj_surj},
the pullback $\pi(\nset)=\delta_k^!\pi_k^\upp(\nset)\colon E\to P$ is then
surjective
in each fiber of $p_k\colon E\to E_{\nset\setminus\{k\}}$. Since the
base map is the identity on $E_{\nset\setminus\{k\}}$, $\pi(\nset)$ is surjective.
\end{proof}

Note that we have proved as well the following result.
\begin{cor}\label{nprojsurjsubm}
In the situation of Theorem \ref{nPullback}, the projection $\pi(\nset)\colon E\to
P$ is a surjective submersion.
\end{cor}

\subsection{Cores of a multiple vector bundle}\label{cores}
Given a double vector bundle $(D,A,B,M)$, the intersection
$(p^D_B)^{-1}(0^B_M)\cap (p^D_A)^{-1}(0^A_M)$ is called the
\textbf{core} of the double vector bundle $(D,A,B,M)$. It has a
natural vector bundle structure over $M$, which is often denoted
$q_C\colon C\to M$. In this section, we explain the cores of multiple
vector bundles.  These cores have also been defined using a different
notation by Alfonso Gracia-Saz and Kirill Mackenzie in
\cite{GrMa12}. 

Let $\E$ be a multiple vector bundle with absolute base $M:=E_\emptyset$. 
For each $S\subseteq \N$ and each $k\in S$, we have the zero section
$0^{\E, S}_{S\setminus \{k\}}\colon E_{S\setminus\{k\}}\to E_{S}$,
$e\mapsto 0^{E_{S}}_e$. For each $R\subseteq S\subseteq \N$,
all compositions of $\# S-\# R$ composable zero sections, starting with some
$0^{R\cup \{i\}}_R\colon E_R\to E_{R\cup\{i\}}$, for some $i\in S\setminus R$,
and ending into $E_S$, are equal and the obtained map is written
$0^{\E, S}_R\colon E_R\to E_S$. In particular, we set $0^{\E, S}_S=\id_{E_S}$. 
If it is clear from the context, which multiple vector bundle 
we are considering, we write $0^S_R:=0^{\E, S}_R$. 
The image of $e\in E_R$ under $0^S_R$ is denoted by
$\nvbzero{S}{e}$, and the image of $E_R$ under $0^S_R$ is written $\nvbzero{S}{R}$.
For better readability we sometimes 
write $\nvbzero{S}{M}:=\nvbzero{S}{\emptyset}$ 
and $\nvbzero{E}{R}:=\nvbzero{\nset}{R}$. 

Choose a subset $S\subseteq \N$ and $j,k\in S$ with $j\neq k$. Then
\[
\begin{tikzcd}
  E_S\ar[r,"p^S_{k}"] \ar[d,"p^S_{j}"] & E_{S\setminus \{k\}}\ar[d,"p^{S\setminus \{k\}}_{j} "]\\
  E_{S\setminus \{j\}} \ar[r,"p^{S\setminus \{j\}}_{k} "] &
  E_{S\setminus \{j,k\}}
\end{tikzcd} 
\]
is a double vector bundle, which has therefore a core
\[E^{S}_{\{j,k\}}:=(p^S_{S\setminus\{j\}})\inv\left(\nvbzero{S\setminus
      \{j\}}{S\setminus \{j,k\}}\right)\cap
  (p^S_{S\setminus\{k\}})\inv\left(\nvbzero{S\setminus
      \{k\}}{S\setminus \{j,k\}}\right)\,.
\]
This core has then an induced
vector bundle structure over $E_{S\setminus \{j,k\}}$ with projection 
$(p^{S\setminus \{j\}}_{S\setminus \{k\}}\circ p^S_{S\setminus \{j\}})\arrowvert_{E^{S}_{\{j,k\}}}$, which we denote
by $c_{\{j,k\}}^S\colon E^{S}_{\{j,k\}}\to E_{S\setminus \{j,k\}}$. 
This is a special case of the side cores, as the following proposition shows. 

\begin{prop}\label{ncoredef}
Let $\E$ be a multiple vector bundle, $S\subseteq \N$ a finite subset
and $J\subseteq S$ non-empty.
The $(S,J)$-core
\[
	E^{S}_{J}:=\bigcap_{j\in J}(p^S_{j})\inv\left(
	\nvbzero{S\setminus \{j\}}{S\setminus J}\right)\,,
\]
is a smooth embedded submanifold of $E_S$ and inherits a vector bundle 
structure over $E_{S\setminus J}$ with projection 
$c_{J}^S:=(\E(S\to S\setminus J))\arrowvert_{E^{S}_{J}}
\colon E^{S}_{J}\to E_{S\setminus J}$. 
In particular, for $J=\{s\}$ of cardinality 1, we get $E^S_J=E_S$ 
and $c^S_J=p^S_s$. 
\end{prop}
\begin{proof}
	That $E^S_J$ is a submanifold of $E_S$ follows from Theorem \ref{nPullback}:
	Consider the $(S,S\setminus J)$-face of $\E$, the $\#J$-fold vector bundle 
	$\E^{S,S\setminus J}$. We denote the corresponding $\#J$-pullback by $P^S_J$. 
	This is the total space of an $\#J$-fold vector bundle $\P^{S}_J$ with 
	absolute base $E_{S\setminus J}$. The image of $E_{S\setminus J}$ under 
	any $\# J$ composable zero sections of $P^S_J$, 
	$Z:=\nvbzero{P^S_J}{E_{S\setminus J}}$ is an embedded submanifold of $P^S_J$. 
	By Corollary \ref{nprojsurjsubm} the $\# J$-fold projection 
	$\pi^S_J\colon E_S\to P^S_J$ is a surjective submersion. $E^S_J$ is 
	the preimage of $Z$ under $\pi^S_J$ and is thus a smooth embedded 
	submanifold of $E_S$. 

	The vector bundle structure is similar to the case $n=2$. Any two elements 
	$e,e'\in E^S_J$ with $c^S_J(e)=c^S_J(e')=:b$ can be added over any 
	$p^S_j$, for $j\in J$, since $p^S_j(e)=0^{S\setminus\{j\}}_b=p^S_j(e')$. 
	All the additions clearly preserve $E^S_J$. For any $j\in J$, 
	$\nvbzero{S\setminus\{j\}}{S\setminus J}$ is an embedded submanifold of
	$E_{S\setminus\{j\}}$ and we get a unique vector bundle structure 
	$E^S_J\to\nvbzero{S\setminus\{j\}}{S\setminus J}$ according to Lemma 
	\ref{emb_prod}. The interchange laws in all the double vector bundles 
	$(E_S,E_{S\setminus\{j_1\}},E_{S\setminus\{j_2\}},E_{S\setminus\{j_1,j_2\}})$ 
	imply that after identification of $\nvbzero{S\setminus\{j\}}{S\setminus J}$ 
	with $E_{S\setminus J}$ all the additions coincide: Since we have
	$\nvbzero{S}{\nvbzero{S\setminus\{j_1\}}{b}}=\nvbzero{S}{b}=
	\nvbzero{S}{\nvbzero{S\setminus\{j_2\}}{b}}$, we find easily 
	\begin{equation}\label{computation_core_add}
		\begin{split}
		e\dvplus{}{S\setminus\{j_1\}}e'&=\left(e\dvplus{}{S\setminus\{j_2\}}
		\nvbzero{S}{\nvbzero{S\setminus\{j_2\}}{b}}\right)
		\dvplus{}{S\setminus\{j_1\}}\left(\nvbzero{S}{\nvbzero{S\setminus\{j_2\}}{b}}
		\dvplus{}{S\setminus\{j_2\}}e'\right)\\
		&=\left(e\dvplus{}{S\setminus\{j_1\}}\nvbzero{S}{\nvbzero{S\setminus\{j_1\}}{b}}\right)
		\dvplus{}{S\setminus\{j_2\}}\left(\nvbzero{S}{\nvbzero{S\setminus\{j_1\}}{b}}
		\dvplus{}{S\setminus\{j_1\}}e'\right)
		=e\dvplus{}{S\setminus\{j_2\}}e'.
		\end{split}
	\end{equation}
	Therefore, $E^S_J$ has a well-defined vector bundle structure over 
	$E_{S\setminus J}$.
\end{proof}

We begin by proving that a side core can be constructed `by stages'.
\begin{lemma}\label{core_of_core}
  Let $\E$ be a multiple vector bundle and $S\subseteq \N$. Choose
  $K\subseteq J\subseteq S$. 
  Then
\begin{equation}\label{core_induction}
  E^S_{J}=\left\{e\in E^S_{K}\mid  p^S_{j}(e)\in \nvbzero{S\setminus \{j\}}{S\setminus J}, j\in J\setminus K, 
	\text{ and }c^S_{K}(e)\in \nvbzero{S\setminus K}{S\setminus J}\right\}\,. 
\end{equation}
\end{lemma}
\begin{proof}
  For simplicity, we denote here by $X$ the set on the
  right-hand side of the equation. First, take $e\in E^S_J$. Then
  since $p^S_j(e)\in \nvbzero{S\setminus \{j\}}{S\setminus J}$ for all
  $j\in J$, and since $K\subseteq J$, we have for $k\in K$:
  $p^S_k(e)=\nvbzero{S\setminus \{k\}}{e_k}$ for some
  $e_k\in E_{S\setminus J}$. Since
  $\nvbzero{S\setminus \{k\}}{e_k}=\nvbzero{S\setminus
    \{k\}}{\nvbzero{S\setminus K}{e_k}}$, we find
  $p^S_k(e)\in \nvbzero{S\setminus \{k\}}{S\setminus K}$ for all
  $k\in K$. Therefore $e\in E^S_K$ with
  $p^S_{j}(e)\in \nvbzero{S\setminus \{j\}}{S\setminus J}$ for
  $j\in J\setminus K$ and we only need to check that
  $c^S_{K}(e)\in \nvbzero{S\setminus K}{S\setminus J}$ in order to
  find that $e\in X$.  But for any choice of $k\in K$, we find
  $c^S_K(e)=p^S_{S\setminus K}(e)=p^{S\setminus\{k\}}_{S\setminus K}(p^S_k(e))=p^{S\setminus \{k\}}_{S\setminus K}(\nvbzero{S\setminus
    \{k\}}{e_k})=0^{S\setminus K}_{e_k}$ with
  $e_k\in E_{S\setminus J}$.  \medskip

  Conversely, take $e\in X$. Then since $e\in E^S_K$ we find for each
  $k\in K$ an element $e_k\in E_{S\setminus K}$ such that
  $p^S_k(e)=0^{S\setminus\{k\}}_{e_k}$. But then
  $e_k=p^{S\setminus\{k\}}_{S\setminus\{K\}}(0^{S\setminus\{k\}}_{e_k})=p^{S\setminus\{k\}}_{S\setminus\{K\}}(p^S_k(e))=p^S_{S\setminus\{K\}}(e)=c^S_K(e)\in
  \nvbzero{S\setminus K}{S\setminus J}$ shows that
  $e\in (p^S_k)\inv\left(\nvbzero{S\setminus \{k\}}{S\setminus
      J}\right)$. Since $k\in K$ was arbitrary and also
  $e\in (p^S_j)\inv\left(\nvbzero{S\setminus \{k\}}{S\setminus
      J}\right)$ for all $j\in J\setminus K$, we find that $e\in E^S_J$.
	\end{proof}

Using this, we prove the following theorem.
\begin{thm}\label{corenvb}
  Let $\E$ be a multiple vector bundle. For each
  $S\subset\N$ and $J\subseteq S$ non-empty, the 
  space $E^{S}_J$ is the total space of an $(\# S-\# J+1)$-fold
  vector bundle in the following way. 
	
  The partition $\rho^S_J=\{J, \{s_1\},\ldots,\{s_{(\# S-\# J+1)}\}\}$
  of $S$ into the set $J$ and sets with one element gives rise to a
  $(\# S-\# J+1)$-cube category $\lozenge^S_J:=\lozenge^{\rho^S_J}$ as
  in section \ref{diagdec}.  We will again write
  $[\nu]:=\cup_{K\in\nu}K$ for any subset $\nu\subseteq\rho^S_J$.  Now
  define $\E^S_J\colon \lozenge^S_J\to\Man$ by setting
  $\E^S_J(\nu)=E^{[\nu]}_J$ if $J\in \nu$ and $\E^S_J(\nu)=E_{[\nu]}$
  if $J\not\in \nu$ and define the morphisms by
	\begin{align*}
		\E^S_J(\nu_1\to \nu_2)&=\E([\nu_1]\to[\nu_2])|_{E^{[\nu_1]}_J}
				\colon E^{[\nu_1]}_J\to E^{[\nu_2]}_J \,, &
				\text{ if }& J\in \nu_2\subseteq \nu_1,\\ 
		\E^S_J(\nu_1\to \nu_2)&=\E([\nu_1]\to [\nu_2])
				\colon E_{[\nu_1]}\to E_{[\nu_2]} \,, &
				\text{ if }& \nu_2\subseteq \nu_1\not\ni J\\		
		\E^S_J(\nu_1\to \nu_2)&=\E([\nu_1]\setminus J\to [\nu_2] )\circ c^{[\nu_1]}_J		
				\colon E^{[\nu_1]}_J\to E_{[\nu_2]} \,, &
				\text{ if }& \nu_2\subseteq \nu_1,\ J\in \nu_1\setminus\nu_2\,.
	\end{align*}
	Then $\E^S_J$ is a $(\# S-\# J+1)$-fold vector bundle.
\end{thm}

\begin{proof}
The nodes of $\E^S_J$ are given by $E^{S'}_J$ for $J\subseteq S'\subseteq S$ 
and $E_I$ for $I\subseteq S\setminus J$. The generating arrows are given
by $p^I_i\colon E_I\to E_{I\setminus\{i\}}$ for $i\in I\subseteq S\setminus J$
and $c^{S'}_J\colon E^{S'}_J \to E_{S'\setminus J}$ and 
$p^{S'}_i|_{E^{S'}_J}\colon E^{S'}_J\to E^{S'\setminus\{i\}}_{J}$ for 
$i\in S'\setminus J$. In the following we just write $p^{S'}_i$
for the restriction $p^{S'}_i|_{E^{S'}_J}$.

For $\# J<\# S$ we prove  by 
induction over $\# J=:l$ that this defines a multiple vector bundle. 
For $J=\{s\}$ of cardinality 1 it is easy to see that $\E^S_J=\E^{S,\emptyset}$,
which is an $\# S$-fold vector bundle by Proposition \ref{sub-k-fold}.

Now assume that $E^S_{\{j_1,\ldots,j_{l-1}\}}$ is the total space of a
$(\# S-l+2)$-fold vector bundle. Choose
$j_l\in S\setminus\{j_1,\ldots,j_{l-1}\}$, $S'\subseteq S$ with
$\{j_1,\ldots,j_l\}=:J\subseteq S'$, and choose $i\in S'\setminus J$. 
Then by the induction hypothesis and
Proposition \ref{sub-k-fold}, {\small\[
\begin{tikzcd}
  E^{S'}_{\{j_1,\ldots,j_{l-1}\}} \ar[rr,"p^{S'}_{j_l}"]\ar[rd,"c^{S'}_{\{j_1,\ldots,j_{l-1}\}}"]\ar[dd,"p^{S'}_i"'] & & E^{{S'}\setminus\{j_l\}}_{\{j_1,\ldots,j_{l-1}\}} \ar[dd,, near start]\ar[rd] & \\
  & E_{{S'}\setminus\{j_1,\ldots,j_{l-1}\}} \ar[rr, crossing over] & & E_{{S'}\setminus\{j_1,\ldots,j_l\}} \ar[dd] \\
  E^{{S'}\setminus\{i\}}_{\{j_1,\ldots,j_{l-1}\}}\ar[rr, near end]\ar[rd,] & & E^{{S'}\setminus\{i,j_l\}}_{\{j_1,\ldots,j_{l-1}\}}\ar[rd] & \\
  & E_{{S'}\setminus\{j_1,\ldots,j_{l-1},i\}} \ar[uu, crossing over, leftarrow]\ar[rr] & & E_{{S'}\setminus\{i,j_1,\ldots,j_l\}}
\end{tikzcd} 
\,,
\]}
is a triple vector bundle, and by \eqref{core_induction}, its upper side core is
\[
\begin{tikzcd}
  E^{S'}_{J}\ar[r,"c^{S'}_{J}"] \ar[d,"p^{S'}_{i}"] & E_{{S'}\setminus J}\ar[d,"p^{{S'}\setminus J}_{i} "]\\
  E^{{S'}\setminus\{i\}}_J \ar[r,"c^{{S'}\setminus \{i\}}_J"] &
  E_{{S'}\setminus (J\cup\{i\})}.
\end{tikzcd} 
\]
Hence this diagram is a double vector bundle (see for example \cite{Mackenzie05b}) 
and, as before, all commutative squares in our
$(\# S-l+1)$-cube diagram are double vector bundles.
\end{proof}

If $l=\# S$, then $J=S$ and $E^S_S$ has a vector bundle structure over
$M$ with projection $c^S_S=\E(S\to \emptyset)$.  The nodes at the
source of only one arrow of $\E^S_J$ are the nodes $E_{\{i\}}$ of $\E$
for $i\in S\setminus J$, and the $(J,J)$-core $c^J_J\colon E^J_J\to M$
of the $\# J$-fold vector bundle bundle $\E^{J,\emptyset}$.

We have then for each $\nu\subseteq \rho^S_J$ an inclusion
$\eta^J(\nu)\colon \E^S_J(\nu)\hookrightarrow E_{[\nu]}$, since
$\E^S_J(\nu)$ is an embedded submanifold of $E_{[\nu]}$ for all 
$\nu\subseteq \rho^S_J$.

\begin{example}\label{decJcore}
Given the $n$-fold vector bundle $\E^\A$ defined in section \ref{dec_n_fold}, 
its $(S,J)$-core $(\E^\A)^S_J$ has nodes
$(\E^\A)^S_J(\nu)=\prod^M_{\nu'\subseteq \nu}A_{[\nu']}$ 
for $\nu\subseteq \rho^S_J:=\rho^S_J$ and can thus be 
identified with $\E^{\mathcal A}_{\rho^S_J}$ defined as in section \ref{diagdec}.
In particular, $(E^\A)^S_S=A_S$.

For instance, for $n=3$ (see Example \ref{decTVB}) we have decomposed cores
\begin{equation*}
	\begin{split}
	E^{\{1,2,3\}}_{\{1,2\}}&=A_{\{3\}}\times_M A_{\{1,2\}}\times_M A_{\{1,2,3\}}\,,\quad 
  E^{\{1,2,3\}}_{\{2,3\}}=A_{\{1\}}\times_M A_{\{2,3\}}\times_M A_{\{1,2,3\}}\,,\\ 
  E^{\{1,2,3\}}_{\{1,3\}}&=A_{\{2\}}\times_M A_{\{1,3\}}\times_M A_{\{1,2,3\}}\,.
	\end{split}
\end{equation*}
\end{example}

\begin{rmk}\label{CoresofFaces}
  \begin{enumerate}
    \item[(a)] Given an $n$-fold vector bundle $\E$ it follows directly from the definitions 
that the cores of the faces of $\E$ are given by the faces of the cores of $\E$. That is,
$(\E^{S,\emptyset})^{S}_J=(\E^S_J)^{\rho^S_J,\emptyset}$ for $J\subseteq S$.
\item[(b)] Note also that \eqref{core_induction} can now be written 
$E^S_J=(E^S_K)^{\rho^S_K}_{\rho^J_K}$.
\item[(c)] For $I,J\subseteq S$ with $I\cap J=\emptyset$ the intersection of the cores 
$E^S_I\cup E^S_J$ is the iterated core 
$(E^S_J)^{\rho^S_J}_{\{\{i\}_{i\in I}\}}=(E^S_I)^{\rho^S_I}_{\{\{j\}_{j\in J}\}}$
\item[(d)] In the case of $I\cup J\neq \emptyset$ the intersection of the core
$E^S_I\cup E^S_J$ is given by $E^S_{I\cap J}$ instead. 
\end{enumerate}
\end{rmk}

\begin{prop}\label{ncoremorph}
Given a morphism $\tau\colon \E\to\F$ of multiple
vector bundles, we have for any $J\subseteq S\subseteq \N$ an 
induced core morphism of the $(\#S-\#J+1)$-fold vector bundles 
$\tau^S_J\colon \E^S_J\to \F^S_J$ defined by 
	\begin{align*}\tau^S_J(\nu)&=\tau([\nu])|_{E^{[\nu]}_J}\colon E^{[\nu]}_J\to F^{[\nu]}_J
          &\text{ for }&\nu\subseteq \rho^S_J \text{ with } J\in\nu\\
		  \tau^S_J(\nu)&=\tau([\nu])\colon E_{[\nu]}\to F_{[\nu]}								
          &\text{ for }&\nu\subseteq \rho^S_J \text{ with } J\not\in \nu\,,
	\end{align*}
where we consider $E^{[\nu]}_J$ and $F^{[\nu]}_J$ as subsets of 
$E_{[\nu]}$ and $F_{[\nu]}$, respectively.
Furthermore, $(\cdot)^S_J$ is a covariant functor from multiple vector bundles to multiple vector bundles.
\end{prop}
\begin{proof}
For $J\not\in\nu$ there is nothing to show as $\E^S_J(\nu)=\E([\nu])$
and $\F^S_J(\nu)=\F([\nu])$ and thus all the maps are well defined
vector bundle morphisms. 

For $J\in\nu$ it remains to be shown that $\tau^S_J$ is well defined,
that is $\tau([\nu])(E^{[\nu]}_J)\subseteq F^{[\nu]}_J$. 
Linearity follows then directly from linearity of $\tau$. 
The manifold $E^{[\nu]}_J$ is defined as the set of all elements of $E_{[\nu]}$ 
that project to $\nvbzero{\E,[\nu]\setminus\{j\}}{[\nu]\setminus J}$ 
for all $j\in J$. Since for all $I\subseteq \nset$, 
$\tau(I)\colon E_{I}\to F_{I}$ is a vector bundle homomorphism 
over $\tau(I\setminus \{i\})$ for all $i\in I$, the image of 
$e\in E^{[\nu]}_J$ under $\tau{[\nu]}$ thus projects to 
$\nvbzero{\F,[\nu]\setminus\{j\}}{[\nu]\setminus J}$ in 
$F_{[\nu]\setminus \{j\}}$ and is an element of $F^{[\nu]}_J$.

Functoriality follows directly from the definition: in the case of $J\not\in\nu$
	\[(\sigma\circ\tau)^S_J(\nu)=(\sigma\circ\tau)([\nu])=
	\sigma([\nu])\circ\tau([\nu])=\sigma^S_J(\nu)\circ\tau^S_J(\nu)\,,\] 
	whereas for $J\in\nu$
	\[(\sigma\circ\tau)^S_J(\nu)=(\sigma\circ\tau)([\nu])|_{E^{[\nu]}_J}=
	\sigma([\nu])|_{F^{[\nu]}_J}\circ\tau([\nu])|_{E^{[\nu]}_J}=\sigma^S_J(\nu)\circ\tau^S_J(\nu)\,.\qedhere\]
\end{proof}

From Theorem \ref{nPullback} we obtain easily the following proposition;
the $n$-fold analogon of the core sequences for double vector bundles,
which were defined by Kirill Mackenzie in \cite{Mackenzie05b}. They are
important in the proof of the existence of decompositions of $n$-fold vector 
bundles. We call them the ultracore sequences of $\E$. 

\begin{prop}\label{nCoreSES}
  Let $\E$ be an $n$-fold vector bundle.
  For each $k\in\nset$, we have a short exact sequence
  \[\begin{tikzcd}
      0\ar[r]& (p^{\nset\setminus\{k\}}_\emptyset)^!E^{\nset}_{\nset}\ar[r,"\iota"] \ar[d] & E\ar[d]\ar[r,"\pi(\nset)"]& P\ar[d]\ar[r]& 0\\
      &E_{\nset\setminus\{k\}}&E_{\nset\setminus\{k\}}&E_{\nset\setminus\{k\}}&
    \end{tikzcd}
  \]
  of vector bundles over $E_{\nset\setminus\{k\}}$, where $P$ is the $n$-pullback defined in Theorem \ref{nPullback}.
\end{prop}

\begin{proof}
  By Theorem \ref{nPullback}, the map $\pi(\nset)\colon E\to P$ is
  a surjective vector bundle morphism over $\id_{E_{\nset\setminus\{k\}}}$.
			
  Take any $e$ in the kernel of $\pi(\nset)$ considered as vector
  bundle morphism over $E_{\nset\setminus\{k\}}$. Denote its
  projection in $E_{J}$ for any $J\subseteq \nset\setminus\{k\}$ by
  $e_{J}$, with $m:=e_{\emptyset}\in M$. Write
  $\nset\setminus\{k\}=\{j_1,\ldots,j_{n-1}\}$.  Define now
  recursively
  \begin{align*}f^0&:=e\,,\qquad  f^l:=f^{l-1}\dvminus{}{\nset\setminus\{j_l\}}\nvbzero{E}{e_{\nset\setminus\{k,j_1,\ldots,j_{l-1}\}}}\,.
  \end{align*}
  Then it is easy to show by induction that
  $p^{\nset}_{I}(f^l)=
  \nvbzero{I}{e_{I\cap(\nset\setminus\{k,j_1,\ldots,j_l\})}}$. 
  The above implies
  that $f^{n-1}$ projects to $\nvbzero{I}{m}$ for all
  $I\subseteq \nset$. It is thus an element of the ultracore
  $E^{\nset}_{\nset}$, and we denote it by $z:=f^{n-1}$.
	
  Now
  \begin{equation}\label{InjCore}
	\begin{split}
		e&=\Biggl(\biggl(\Bigl(z
    \dvplus{}{\nset\setminus\{j_{n-1}\}}\nvbzero{E}{e_{\{j_{n-1}\}}}\Bigr)
    \dvplus{}{\nset\setminus\{j_{n-2}\}}\nvbzero{E}{e_{\{j_{n-1},j_{n-2}\}}}\biggr)
    \dvplus{}{\nset\setminus\{j_3\}}\ldots\Biggr)
                   \dvplus{}{\nset\setminus\{j_1\}}\nvbzero{E}{e_{\nset\setminus\{k\}}}\\
                 &=:\iota(z,e_{\nset\setminus\{k\}})\,.
  \end{split}
	\end{equation}  
  and the defined map
  $\iota\colon E^\nset_\nset\times_M E_{\nset\setminus\{k\}}\to E$ is
  clearly an injective morphism of vector bundles over
  $E_{\nset\setminus\{k\}}$, making the sequence exact. We let the
  reader check that $\iota$ does not depend on the chosen order of the
  set $\nset\setminus\{k\}$.
\end{proof}

\section{Splittings of $n$-fold vector bundles}\label{sec_ex_splitting}
In this section we achieve our main goal in this paper: we prove that
any $n$-fold vector bundle admits a (non-canonical) linear
splitting. We begin by discussing the notions of linear splitting
versus linear decomposition. Then we prove inductively our main
theorem, and finally we explain how $n$-fold vector bundles can now be
defined using \emph{$n$-fold vector bundle atlases}.

 \subsection{Splittings and decompositions of $n$-fold vector bundles}
	\label{splitdecnvb}

        Let $\E$ be an $n$-fold vector bundle. This gives rise to a
        family $\A$ of smooth vector bundles
        $\A=(q_J\colon A_J\to M)_{J\subseteq\nset,\,\# J<\infty}$ over
        $M=\E(\emptyset)$ defined by $A_{\{i\}}=E_{\{i\}}$ for
        $i=1,\ldots,n$ and $A_J=E^J_J$ for $\# J\geq 2$. By Example
        \ref{decJcore}, if $\E$ is already a decomposed $n$-fold
        vector bundle, then each element of the family of vector
        bundles defining it appears as one of the cores of $\E$.  This
        is why we call the vector bundles $A_J=E^J_J$ the \textbf{building
          bundles of $\E$}.

We can then consider the decomposed $n$-fold vector bundles $\E^\A$ 
and $\overline{\E}:=\E^{\overline{\A}}$ defined in 
Section \ref{dec_n_fold}. We call $\E^\A$ the decomposed 
$n$-fold vector bundle associated to $\E$ and $\overline{\E}$ the 
vacant, decomposed $n$-fold vector bundle associated to $\E$.

\begin{mydef}\label{def_n-dec}
A \textbf{linear splitting} of the $n$-fold vector bundle
$\E$ is a monomorphism
$\Sigma\colon\overline{ \E}\to \E$ of $n$-fold vector bundles,
such that for $i=1,\ldots,n$, $\Sigma(\{i\})\colon E_{\{i\}}\to E_{\{i\}}$ is the
identity.

A \textbf{decomposition} of the $n$-fold vector bundle $\E$ is a
natural isomorphism $\mathcal S\colon \E^{\mathcal A}\to \E$
of $n$-fold vector bundles over the identity maps
$\mathcal S(\{i\})=\id_{E_{\{i\}}}\colon A_{\{i\}}\to E_{\{i\}}$ such that
additionally the induced core morphisms $\S^I_I(\{I\})$ are the
identities $\id_{E^I_I}$ for all $I\subseteq\nset$.
\end{mydef}

Linear splittings and decompositions of double vector bundles are
equivalent to each other.  Given a splitting $\Sigma$, define the
decomposition by
$\mathcal{S}(a_m,b_m,c_m):=\Sigma(a_m,b_m)+_B(0_{b_m}^D+_A c_m)
=\Sigma(a_m,b_m)+_A(0_{a_m}^D+_B c_m)$.  Conversely, given a
decomposition $\mathcal{S}$ define the splitting by
$\Sigma(a_m,b_m):=\mathcal{S}(a_m,b_m,0_m^C)$.  These two
constructions are obviously inverse to each other. 
We prove here that a similar equivalence holds true in the general 
case of $n$-fold vector bundles. 

\medskip

A linear splitting $\Sigma$ of an $n$-fold vector bundle $\E$ and
decompositions $\S^{I}$ of the highest order cores -- the $(n-1)$-fold
vector bundles $\E^{\nset}_{I}$ for all $I\subseteq\nset$ with
$\# I=2$ -- are called \textbf{compatible} if they coincide on all
possible intersections.  That is,
$\S^I(\{\{k\}_{k\in\nset\setminus I}\})|_{\overline{\E}(\nset\setminus
  I)}= \Sigma(\nset\setminus I)$
and
$\S^I(\rho^{\nset}_{I})|_{(E^\A)^{\nset}_I\cap (E^\A)^\nset_J}=
\S^J(\rho^{\nset}_{J})|_{(E^\A)^{\nset}_I\cap (E^\A)^\nset_J}$
for all $I,J\subseteq\nset$ of cardinality 2. Note that we view here
the total spaces of $(\E^\A)^{\nset}_I$, $(\E^\A)^{\nset}_J$ and of
$\E^{\nset}_I$ and $\E^\nset_J$ as embedded in $E^\A=\E^\A(\nset)$ and
$E=\E(\nset)$, respectively.  Also recall that
$\E^\nset_I(\{k\}_{k\in\nset\setminus I})=\E(\nset\setminus I)$ by
definition.

\begin{thm}\label{thm_split_dec_n}
\begin{enumerate}
\item[(a)] Let $\mathcal S$ be a decomposition of an $n$-fold vector 
bundle $\E\colon\square^n\to \Man$. Then the composition 
$\Sigma=\mathcal S\circ\iota\colon\overline{\E}\to\E$, with 
$\iota$ defined as in \eqref{emb_spl_dec}, is a splitting of $\E$. 
Furthermore, the core morphisms $\S^\nset_J\colon\E^{\rho_{J}}\to\E^{\nset}_J$
are decompositions of $\E^{\nset}_J$ for all $J\subseteq\nset$ and 
these decompositions and the linear splitting are compatible.

\item[(a)] Conversely, given a linear splitting $\Sigma$ of $\E$ and 
	compatible decompositions of the highest order cores $\E^{\nset}_J$ 
    with top maps $\S^J\colon (E^\A)^\nset_J\to E^{\nset}_J$, 
    for $J\subseteq \nset$ with $\# J=2$, there 
	exists a unique decomposition $\mathcal S$ of $\E$ such that 
	$\Sigma=\mathcal S\circ\iota$ and such that the core morphisms 
	of $\S$ are given by $\S^\nset_J(\rho^\nset_J)=\S^J$ for all $J$. 
\end{enumerate}
\end{thm}

\begin{proof}
  Let us consider a decomposition $\S\colon \E^\A\to\E$.  Then the
  composition $\Sigma=\mathcal S\circ\iota$ is clearly a monomorphism
  of $n$-fold vector bundles, with
  $\Sigma(\{i\})=\S(\{i\})\circ\iota(\{i\})=\id_{E_{\{i\}}}\circ
  \id_{E_{\{i\}}}=\id_{E_{\{i\}}}$. Furthermore, Proposition
  \ref{ncoremorph} implies that the restrictions $\S^\nset_J$ are
  isomorphisms of multiple vector bundles. Since for any 
  $\nu\subseteq\rho^\nset_J$ the $(\nu,\nu)$-core of $\E^\nset_J$ 
  equals $E^{[\nu]}_{[\nu]}$ which follows from Remark 
  \ref{CoresofFaces} for $J\in\nu$ and directly from the definition
  for $J\not\in \nu$, these are all the building bundles of $\E^\nset_J$.
  Now $\S^{[\nu]}_{[\nu]}=\id_{E^{[\nu]}_{[\nu]}}$ and thus
  $\S^\nset_J$ induces the identity on all building bundles of
  $\E^\nset_J$ and is therefore a decomposition. Since all $\S^\nset_J$ 
  and $\Sigma$ are defined as restrictions of the same map $\S$ they 
  are clearly compatible. 
\medskip

Conversely, assume that we have a splitting $\Sigma$ of $\E$ 
and compatible decompositions $\S^J$ of the cores $\E^\nset_J$ 
with $J\subseteq \nset$, $\# J=2$ as in (b). We prove that there 
is a unique decomposition $\S$ of $\E$ that restricts in the sense 
of (b) to $\Sigma$ and the $\S^J$. 

Let now $J_1, \ldots,J_{\binom{n}{2}}$ denote the subsets of $\nset$ 
with $\# J_k=2$. 
We define now an increasing chain of $\binom{n}{2}$ decomposed $n$-fold 
vector bundles as follows. For $k=0,\ldots,\binom{n}{2}$ define
a family of vector bundles over $M$, $\A^k=(B_I)_{I\subseteq\nset}$ 
with $B_I=A_I$ for all $I$ with either $\# I=1$ or if there is 
$i\leq k$ such that $J_i\subseteq I$; and $B_I=M$ otherwise. Now let
$\E^k:=\E^{\A^k}$ with total space $E^k:=\E^{\A^k}(\nset)$. There 
are obvious inclusions 
$\overline{\E}(\nset)=E^0\hookrightarrow E^1\hookrightarrow\ldots
\hookrightarrow E^{\binom{n}{2}}=\E^\A(\nset)$. We thus view the $E^k$ as  
submanifolds of $\E^\A(\nset)$. Note that additionally 
$(E^\A)^{\nset}_{J_i}\subseteq E^k$ for all $i\leq k$. 
Now we show that we can define a decomposition $\S$ of $\E$ 
inductively on the $E^k$ for $k=0,\ldots,\binom{n}{2}$ and that 
it is unique with respect to the given linear splittings.

Since $\E^0=\overline{\E}$ we set $\S^0:=\Sigma$ and this is clearly
unique in the sense of (b). By the compatibility condition it also
restricts to $\S^{J_i}$ on $E^0\cap (E^\A)^\nset_{J_i}$ 
for $i=1,\ldots,\binom{n}{2}$.  Take
now $k\geq 0$ and assume that we have a uniquely defined injective
morphism of $n$-fold vector bundles $\S^k\colon\E^k\to \E$ that
restricts to $\Sigma$ on $E^0$ and to $\S^{J_i}$ on
$E^k\cap(E^\A)^\nset_{J_i}$ for $i=1,\ldots,\binom{n}{2}$.
Take $\mathbf{x}=(a_I)_{I\subseteq\nset}\in E^{k+1}$. Then in 
particular $a_I=0_m^{A_I}$ if $\#I\geq 2$ and there is no $i\leq k+1$ 
with $J_i\subseteq I$. Set $\mathbf{y}:=(b_I)_{I\subseteq\nset}$ with 
$b_I=a_I$ if either $\# I=1$ or there is $i\leq k$ such that 
$J_i\subseteq I$ and $b_I=0^{A_I}_m$ otherwise. Set furthermore 
$\mathbf{z}:=(c_I)_{I\subseteq\nset}$ where $c_I=b_I$ whenever 
$I\subseteq\nset\setminus J_{k+1}$, $c_I=a_I$ whenever $J_{k+1}\subseteq I$ 
and there is no $i\leq k$ with $J_i\subseteq I$, and $c_I=0^{A_I}_m$ otherwise. 
Then $\mathbf{y}\in E^k$ and $\mathbf{z}\in (\E^\A)^\nset_{J_{k+1}}$. 
Furthermore, writing $J_{k+1}=\{s,t\}$, it is easy to check that 
\[
\mathbf{x}=\mathbf{y}\dvplus{}{\nset\setminus\{s\}}\left(
\nvbzero{\nset}{p_s(\mathbf{y})}\dvplus{}{\nset\setminus\{t\}}
\mathbf{z}\right)=\mathbf{y}\dvplus{}{\nset\setminus\{t\}}\left(
\nvbzero{\nset}{p_t(\mathbf{y})}\dvplus{}{\nset\setminus\{s\}}
\mathbf{z}\right)\,.
\] 
The last equality follows directly from the interchange law in the double vector bundle 
$(E;E_{\nset\setminus\{s\}},E_{\nset\setminus\{t\}};E_{\nset\setminus\{s,t\}})$
since $\S^{J_{k+1}}(\mathbf{z})$ is in the core of this double vector bundle. 
Thus we can define 
\[
\S^{k+1}(\mathbf{x}):=\S^k(\mathbf{y})\dvplus{}{\nset\setminus\{s\}}\left(
\nvbzero{\nset}{p_s\bigl(\S^k(\mathbf{y})\bigr)}\dvplus{}{\nset\setminus\{t\}}
\S^{J_{k+1}}(\mathbf{z})\right)=\S^k(\mathbf{y})\dvplus{}{\nset\setminus\{t\}}\left(
\nvbzero{\nset}{p_t\bigl(\S^k(\mathbf{y})\bigr)}\dvplus{}{\nset\setminus\{s\}}
\S^{J_{k+1}}(\mathbf{z})\right)\,.
\]
It is easy to check that this defines an injective morphism of 
$n$-fold vector bundles $\S^{k+1}\colon \E^{k+1}\to \E$. Linearity
over $E_{\nset\setminus\{j\}}$ follows directly from linearity of $\S^k$ 
and $\S^{J_{k+1}}$ and the interchange laws in the double vector bundles
$(E;E_{\nset\setminus\{j\}},E_{\nset\setminus\{s\}};E_{\nset\setminus\{j,s\}})$
and 
$(E;E_{\nset\setminus\{j\}},E_{\nset\setminus\{t\}};E_{\nset\setminus\{j,t\}})$
since the construction of $\mathbf{y}$ and $\mathbf{z}$ from $\mathbf{x}$ 
is linear.
If now $\mathbf{x}$ was already in $E^k$, then $\mathbf{y}=\mathbf{x}$ 
and thus $\S^{k+1}$ restricts to $\S^k$ on $E^k$ and therefore also to 
$\Sigma$. If $\mathbf{x}$ was in $(E^\A)^\nset_{J}$ for any $J\subseteq\nset$
with $\#J=2$, then $\mathbf{y}\in E^k\cap(E^\A)^\nset_J$
and by induction hypothesis $\S^k(\mathbf{y})=\S^{J}(\mathbf{y})$. 
Furthermore, $\mathbf{z}\in (E^\A)^\nset_{J_{k+1}}\cap(E^\A)^\nset_J$ 
and by the compatibility of $\S^{J_{k+1}}$ with $\S^J$ we get that
$\S^{J_{k+1}}(\mathbf{z})=\S^J(\mathbf{z})$. Thus clearly $\S^{k+1}$ 
restricts to all $\S^J$ on the intersection $E^{k+1}\cap (E^\A)^\nset_J$. 
Also it clearly is the only morphism from $\E^{k+1}$ to $\E$ 
restricting to $\S^{k}$ on $E^{k}$ and to all $\S^{J}$ and thus by 
the induction hypothesis the only morphism restricting to $\Sigma$ and 
all $\S^{J}$.
Thus we find eventually a unique injective morphism 
$\S:=\S^{\binom{n}{2}}\colon \E^\A\to \E$ that restricts to $\Sigma$ 
and all $\S^J$ for $\# J=2$. That $\S$ is surjective now follows from 
linearity and a dimension count. 
\end{proof}

\subsection{Existence of splittings}\label{ex_splitting}
In this section, we finally state and prove our main theorem.
We prove by induction that every $n$-fold vector
bundle is non-canonically isomorphic to a decomposed one. 
\begin{thm}\label{thm_dec_n_vb}
  Let $\E$ be an $n$-fold vector bundle. Then there is a linear
  splitting
  \[\Sigma\colon \overline{\E} \to \E\,,
    \]
    that is a monomorphism of $n$-fold vector bundles from the vacant,
    decomposed $n$-fold vector bundle $\overline{\E}$ associated to $\E$, 
		which was defined in Section \ref{splitdecnvb}, into $\E$.

  \end{thm}

  \begin{proof}
We prove the following two claims by induction over $n$.
\begin{enumerate}
\item[(a)] Given an $n$-fold vector bundle $\E$, there exist $n$
  linear splittings $\Sigma_{\nset\setminus\{k\}}$ of
  $\E^{\nset\setminus\{k\}, \emptyset}$ for $k\in\nset$, such that
  $\Sigma_{\nset\setminus\{i\}}(I)=\Sigma_{\nset\setminus\{j\}}(I)$
  for any $I\subseteq\nset\setminus\{i,j\}$.
\item[(b)] Given a family of splittings as in (a), there exists a linear
  splitting of $\E$ with $\Sigma(I)=\Sigma_{\nset\setminus\{k\}}(I)$
  whenever $I\subseteq\nset\setminus\{k\}$.
\end{enumerate}

The case of $n=1$ is trivial. Take now $n\geq 2$ and assume that
both statements are true for $l$-fold vector bundles, for $l<n$.
First, we prove (a).  This is equivalent to having splittings
$\Sigma_I$ of $\E^{I,\emptyset}$ for all $I\subsetneq\nset$ such that
$\Sigma_{I_1}(J)=\Sigma_{I_2}(J)$ whenever $J\subseteq I_1\cap
I_2$. We prove that claim with an induction over $\sharp I$. For all
$I\subseteq \nset$ with $\sharp I=1$ or $\sharp I=2$, this is
immediate.
								
Assume now that we have fixed linear splittings of $\E^{I,\emptyset}$ 
for all $I$ with $\#I =l\leq n-2$, such that for all
$J\subseteq I_1\cap I_2$, $\Sigma_{I_1}(J)=\Sigma_J(J)=\Sigma_{I_2}(J)$.  
For any $I\subsetneq\nset$ with $\#I=l+1$ we can then find by induction
hypothesis (b) a linear splitting $\Sigma_I$ of $\E^{I,\emptyset}$ which 
satisfies $\Sigma_I(J)=\Sigma_J(J)$ for all $J\subseteq I$.  Now for 
$I_1,I_2$ of cardinality $l+1$ and $J\subseteq I_1\cap I_2$, we get
$\Sigma_{I_1}(J)=\Sigma_J(J)=\Sigma_{I_2}(J)$.  This shows that part
(a) is satisfied for every $n$-fold vector bundle since we eventually
find linear splittings $\Sigma_{\nset\setminus\{k\}}$ of all
$\E_{\nset\setminus\{k\}}$ which agree on all subsets
$I\subseteq \nset$ of cardinality $\# I\leq (n-2)$.
								
We denote in the following their top maps by
    \[\Sigma_k:=\Sigma_{\nset\setminus\{k\}}(\nset\setminus\{k\})\colon 
      \prod_{i\in\nset\setminus\{k\}}^M E_{\{i\}}\to E_{\nset\setminus
        \{k\}}\,.\] It is easy to check that given $m\in M$ and
    $e_i\in E_{\{i\}}$ with $p^{\{i\}}_\emptyset(e_i)=m$ for
    $i=1,\ldots,n$, the tuple
    $(\Sigma_1(e_2,\ldots,e_n),\Sigma_2(e_1,e_3,\ldots,e_n),\ldots,\Sigma_n(e_1,\ldots,e_{n-1})$
    is an element of $P$.  Short exact sequences of vector bundles are
    always non-canonically split, so we can take a splitting
    $\theta_1$ of the short exact sequence of vector bundles over
    $E_{\nset\setminus \{1\}}$ in Proposition \ref{nCoreSES}.  Define
    $\Sigma^E_1\colon \prod_{i\in\nset}^M E_{\{i\}}\to E$ by
    \begin{equation}\label{sigma1}
				\Sigma^E_1\colon (e_1,\ldots,e_n)\mapsto \theta_1\bigl(\Sigma_1(e_2,\ldots,e_n),\Sigma_2(e_1,e_3,\ldots,e_n),\ldots,\Sigma_n(e_1,\ldots,e_{n-1})\bigr)\,.
			\end{equation}
                        This is a vector bundle morphism over the
                        linear splitting $\Sigma_1$ of
                        $E_{\nset\setminus \{1\}}$ such that
		\begin{equation}\label{correctprojections}
		p_j\bigl(\Sigma^E_1(e_1,\ldots,e_n)\bigr)=\Sigma_j(e_1,\ldots,\hat{e_j},\ldots,e_n)\in E_{\nset\setminus\{j\}}
              \end{equation}
              for $j=2,\ldots,n$.
		However, $\Sigma^E_1$ is not necessarily linear over $\Sigma_j$ 
		as $\theta_1$ is not a morphism of $n$-fold vector bundles. We will 
		inductively construct a	morphism which is linear over all sides. 
		
		First we do this locally: we choose a neighbourhood
                $U$ of $m\in M$ that trivialises each of the $E_{\{i\}}$,
                for $i=1,\ldots,n$. Fix smooth local frames
                $(b_i^1,\ldots,b_i^{l_i})$ of $E_{\{i\}}$ for $l_i=\rk
                E_{\{i\}}$. Every element of $\prod_{i\in\nset}^M E_{\{i\}}$ over
                $m\in U$ can thus be written uniquely as
			\[(e_1,\ldots,e_n)=\Bigl(\sum_{j=1}^{l_1}\beta_1^{j} b^{j}_1(m),\ldots,\sum_{j=1}^{l_n}\beta_n^{j} b^{j}_n(m)\Bigr)\]
		where $\beta_i^{j}\in \R$. Assume now that we have a morphism 
		$\Sigma^E_{k,U}\colon\overline{\E}|_U\to \E|_U$ which is linear over the splittings 
		$\Sigma_j$ for $j=1,\ldots k$ and satisfies additionally 
		\eqref{correctprojections} for all other $j$. We then define $\Sigma^E_{k+1,U}$ by
			\[\Sigma^E_{k+1,U}(e_1,\ldots,e_n):= \sum_{j=1,\ldots,l_{k+1}}^{E\to E_{\nset\setminus\{k+1\}}}\beta^{j}_{k+1}\dvtimes{}{\nset\setminus{\{k+1\}}}\Sigma^E_{k,U}\bigl(e_1,\ldots,e_k,b_{k+1}^{j}(m),e_{k+2},\ldots,e_n\bigr)\,.\]
		That this morphism is still a vector bundle morphism over $\Sigma_j$ for all
		$j=1,\ldots,k$ follows from the interchange laws in the double vector bundles 
		$(E,E_{\nset\setminus\{j\}},E_{\nset\setminus\{k+1\}},E_{\nset\setminus\{j,k+1\}})$.
		That it is also a vector bundle morphism over $\Sigma_{k+1}$ is immediate. 
		It furthermore still satisfies \eqref{correctprojections} for all other $j$.
		Starting with the restriction to $U$ of $\Sigma^E_1$ from \eqref{sigma1} 
		we get after $(n-1)$ iterations the top map of a local linear splitting 
		$\Sigma^E_U$ of $\E|_U$. 

                Now we will prove the existence of a global splitting
                using a partition of unity. This method was already
                given for double vector bundles in the original
                reference by Pradines \cite{Pradines77}.  Choose a locally finite cover of
                neighbourhoods as above,
                $\mathcal{U}=\{U_\alpha\}_{\alpha\in A}$, and a
                partition of unity $\{\varphi_\alpha\}_{\alpha\in A}$
                subordinate to $\mathcal{U}$.  Take then the local
                linear splittings $\Sigma^E_{U_\alpha}$ and define the
                global splitting for $(e_1,\ldots,e_n)$ over $m\in M$
                by
		\[\Sigma^E(e_1,\ldots,e_n):=\sum^{E\to E_{\nset\setminus\{1\}}}_{\{\alpha\colon m\in U_\alpha\}}\varphi_\alpha(m)\dvtimes{}{\nset\setminus\{1\}}\Sigma^E_{U_\alpha}(e_1,\ldots,e_n)\,.\]
		That this is a vector bundle morphism over all
                $\Sigma_j$ follows from simple computations using
                again the interchange laws in the double vector
                bundles
                $(E,E_{\nset\setminus\{1\}},E_{\nset\setminus\{j\}},E_{\nset\setminus\{1,j\}})$.
                Injectivity follows directly from this as all
                $\Sigma_k$ are injective.
                The linear splitting is then given by $\Sigma(\nset):=\Sigma^E$ and
		$\Sigma(I):=\Sigma_{\nset\setminus\{k\}}(I)$ whenever 
		$I\subseteq\nset\setminus \{k\}$. This completes the proof.
	\end{proof}

	\begin{cor}
	Every $n$-fold vector bundle $\E$ is non-canonically isomorphic to the
	associated decomposed $n$-fold vector bundle defined in Section \ref{dec_n_fold}.	
	\end{cor}
	\begin{proof}
          This follows from Theorem \ref{thm_dec_n_vb} and Theorem
          \ref{thm_split_dec_n}. To apply Theorem \ref{thm_split_dec_n} we
          have to show that we can construct compatible decompositions of
          all the highest order cores. This follows from a similar argument
          to the beginning of the proof of Theorem \ref{thm_dec_n_vb}. 
          
          We have to consider all iterated highest order cores. These are 
          firstly the $(n-1)$-fold vector bundles $\E^\nset_I$ with
          $I\subseteq\nset$ and $\# I=2$, secondly the $(n-2)$-fold 
          vector bundles $(\E^\nset_I)^{\rho^\nset_I}_{\nu}$ with $\nu\subseteq\rho^\nset_I$ 
          and $\# \nu=2$ and so forth. Theorem \ref{thm_dec_n_vb} lets us choose 
          linear splittings of all these multiple vector bundles. Note that the 
          same multiple vector bundles can occur multiple times (see for example 
          Remark \ref{CoresofFaces} (c)). For these we still fix only one linear 
          splitting. With Theorem 
          \ref{thm_split_dec_n} we obtain then firstly unique decompositions of 
          all occurring double vector bundles. After fixing these, with Theorem 
          \ref{thm_split_dec_n} we obtain decompositions of all occurring triple 
          vector bundles and these are all compatible by construction. 
          Fixing these we obtain compatible decompositions of all occurring $4$-fold
          vector bundles and so forth. Eventually after obtaining compatible 
          decompositions of the highest order cores Theorem \ref{thm_split_dec_n}
          gives us a decompositions of $\E$. 
	\end{proof}
	\begin{cor}\label{splitCoreSES}
	For every $n$-fold vector bundle $\E$ and the associated $n$-pullback $\P$
	there is an injective morphism of $n$-fold vector bundles 
	$\Sigma^P\colon \P\to\E$ simultaneously splitting all the ultracore 
	sequences from Proposition \ref{nCoreSES}.
	\end{cor}
	\begin{proof}
          We can choose a decomposition of $\E$ with top map
          $\S^E\colon \E^\A(\nset)\to E$. This is a morphism over
          decompositions of the faces $E_{\nset\nok}$ for all
          $k\in\nset$.  These decompositions induce a canonical
          associated decomposition of $\P$, the top map of which we
          denote by
          $\S^P\colon
          \prod^M_{I\subsetneq\nset}E^I_I\to
          P$.  Together with the canonical inclusion
          $\iota \colon
          \prod^M_{I\subsetneq\nset}E^I_I\to
          \E^\A(\nset)$ we then define such a splitting with top map
          given by $\Sigma^P(\nset):=\S^E\circ\iota\circ(\S^P)\inv$.
	\end{proof}

        \subsection{$n$-fold vector bundle atlases}\label{sec_atlas}
In this section we show how a change of splittings corresponds to 
statomorphisms of the decomposed multiple vector bundle, which were 
introduced in \cite{GrMa12}.
We then explain how $n$-fold vector bundles can
alternatively be defined using smoothly compatible $n$-fold vector
bundle charts.

For $I$ a finite subset of $\N$, we denote by
$\mathcal P(I)=\{\{I_1,\ldots,I_k\}\mid I=I_1\sqcup\ldots\sqcup I_k\}$
the set of disjoint partitions of $I$.  Since the elements of
$\mathcal{P}(I)$ are sets, not tuples, we do not take the order into
account. That is, we do not distinguish  the partition
$\{I_1,I_2\}$ from $\{I_2,I_1\}$.

\begin{mydef}
	Let $\E$ be an $n$-fold vector bundle. A \textbf{statomorphism} of
	$\E$ is an isomorphism $\tau\colon \E\to \E$ that induces
	the identity on all building bundles $E^I_I$ for $I\subseteq \nset$.
	The set of statomorphisms of $\E$ forms a group with composition. 
\end{mydef}

\begin{prop}\label{dec_torsor_stato}
  Let $\E$ be an $n$-fold vector bundle and $\E^\A$ the corresponding
  decomposed $n$-fold vector bundle as in Definition
  \ref{def_n-dec}. The set of global decompositions of $\E$ is a
  torsor over the group of statomorphisms of $\E^\A$.
\end{prop}
\begin{proof}
	Given a decomposition $\S\colon \E^\A\to \E$ and a statomorphism
	$\tau\colon \E^\A\to \E^\A$ the composition 
	$\S\circ\tau\colon \E^\A\to \E$ is again a decomposition of $\E$.
	This defines a right action of the group of statomorphisms of $\E^\A$
	onto the set of decompositions of $\E$. Given two decompositions 
	$\S_1,\S_2\colon \E^\A\to \E$ the composition 
	$\tau:=\S_1\inv\circ\S_2\colon \E^\A\to \E^\A$ defines a
	statomorphism of $\E^\A$ such that $\S_1\circ\tau=\S_2$. This shows 
	that the action is transitive. That it is free is immediate as 
	$\S\circ\tau=\S$ clearly implies $\tau=\id$. 
\end{proof}

The following description of statomorphisms can be found in slightly 
different notation in \cite{GrMa12}. 
\begin{prop}
	A statomorphism $\tau$ of $\E^\A$ is necessarily of the following form:
	\begin{equation}\label{statodec}
		\tau(\nset)\colon (e_I)_{I\subseteq\nset}\mapsto \left(
		\sum_{\rho=\{I_1,\ldots,I_k\}\in\mathcal{P}(I)}
		\varphi_{\rho}(e_{I_1},\ldots,e_{I_k})
		\right)_{I\subseteq\nset}\,,
	\end{equation}
	where $\varphi_\rho\in\Gamma(\Hom(E_{I_1}^{I_1}\otimes\ldots
	\otimes E_{I_k}^{I_k},E^I_I))$ and for the trivial partition $\rho=\{I\}$ 
	we additionally demand $\varphi_{\{I\}}=\id_{E^I_I}$. 
\end{prop}

Now we define $n$-fold vector bundle charts and atlases and show that
our definition of $n$-fold vector bundles is equivalent to the definition
in terms of charts. 

\begin{mydef}\label{n_fold_atlas}
  Let $M$ be a smooth manifold and $E$ a topological space
  together with a continuous map $\Pi\colon E\to M$.  An
  \textbf{n-fold vector bundle chart} is a tuple
  \[c=(U,\Theta,(V_I)_{I\subseteq \nset}),\]
  where $U$ is an open set in $M$, for each $I\subseteq\nset$ the
  space $V_I$ is a (finite dimensional) real vector space and
  $\Theta\colon \Pi\inv(U)\to U \times \prod_{I\subseteq \nset}V_I$ is
  a homeomorphism such that $\Pi=\pr_1\circ\Theta$.

  Two $n$-fold vector bundle charts
  $c=(U,\Theta, (V_I)_{I\subseteq\nset})$ and
  $c'=(U',\Theta', (V_I')_{I\subseteq\nset})$ are \textbf{smoothly
    compatible} if $V_I=V_I'$ for all $I\subseteq\nset$ and the 
		``change of chart'' $\Theta'\circ\Theta\inv$
  over $U\cap U'$ has the following form:
\begin{equation}\label{change_of_charts}
  \bigl(p,(v_I)_{I\subseteq\nset}\bigr)\mapsto\left(p,\left(\sum_{\rho=\{I_1,\ldots,I_k\}\in\mathcal P(I)}\omega_{\rho}(p)(v_{I_1},\ldots,v_{I_k})\right)_{I\subseteq\nset}\right)
\end{equation}
with $p\in U\cap U'$, $v_I\in V_I$ and
$\omega_\rho\in C^\infty(U\cap U',\Hom(V_{I_1}\otimes\ldots\otimes
V_{I_k},V_{I}))$ for $\rho=\{I_1,\ldots,I_k\}\in\mathcal P(I)$.

A \textbf{smooth n-fold vector bundle atlas} $\lie A$ on $E$
is a set of n-fold vector bundle charts of $E$ that are
pairwise smoothly compatible and such that the set of underlying open
sets in $M$ covers $M$.  As usual, $E$ is then a
smooth manifold and two smooth $n$-fold vector bundle atlases $\lie A_1$
and $\lie A_2$ are \textbf{equivalent} if their union is a smooth
n-fold vector bundle atlas.  A smooth \textbf{$n$-fold vector bundle} structure
on $E$ is an equivalence class of smooth n-fold vector bundle
atlases on $E$.
\end{mydef}

Let $\mathbb E$ be an $n$-fold vector bundle. By Theorem
\ref{thm_dec_n_vb} and Theorem \ref{thm_split_dec_n} we have a
decomposition $\mathcal S\colon \E^{\mathcal A}\to \E$ of
$\E$, with $\mathcal A$ the family $(A_I)_{I\subseteq \nset}$ of
vector bundles over $M$ defined by $A_{\{i\}}=\E(\{i\})$ for
$i\in\nset$ and $A_I=E^I_I$ for $I\subseteq \nset$, $\# I\geq 2$. 
Set $\Pi=\E(\nset\to\emptyset)\colon E\to M$. For
each $I\subseteq \nset$, set $V_I:=\mathbb R^{\dim A_I}$, the vector
space on which $A_I$ is modelled. Take a covering
$\{U_\alpha\}_{\alpha\in\Lambda}$ of $M$ by open sets trivialising all
the vector bundles $A_I$;
\[ \phi^\alpha_I\colon q_{I}^{-1}(U_\alpha)\overset{\sim}{\longrightarrow} U_\alpha\times V_I
\]
for all  $I\subseteq \nset$ and all $\alpha\in\Lambda$.
Then we define $n$-fold vector bundle charts $\Theta_\alpha\colon\Pi^{-1}(U_\alpha)\to U_\alpha\times \prod_{I\subseteq \nset}V_I$ by 
\[\Theta_\alpha=\left(\Pi\times \left(\phi^\alpha_I\right)_{I\subseteq \nset}\right)\circ \mathcal S^{-1}\an{\Pi^{-1}(U_\alpha)}.
\]
Given $\alpha,\beta\in\Lambda$ with $U_\alpha\cap U_\beta\neq\emptyset$, the change of chart
\[\Theta_\alpha\circ\Theta_\beta\inv\colon (U_\alpha\cap U_\beta)\times \prod_{I\subseteq \nset}V_I\to  (U_\alpha\cap U_\beta)\times \prod_{I\subseteq \nset}V_I
\]
is given by 
\begin{equation}\label{simple_change_of_charts}
(p, (v_I)_{I\subseteq \nset})\mapsto (p, (\rho^{\alpha\beta}_I(p)v_I)_{I\subseteq \nset}),
\end{equation}
with
$\rho_I^{\alpha\beta}\in C^\infty(U_\alpha\cap U_\beta,
\operatorname{Gl}(V_I))$
the cocycle defined by $\phi^\alpha_I\circ (\phi_I^\beta)\inv$. The
two charts are hence smoothly compatible and we get an $n$-fold vector
bundle atlas
$\lie A=\{(U_\alpha,\Theta_\alpha, (V_I)_{I\subseteq\nset})\mid \alpha\in
\Lambda\}$ on $E$.

\medskip

Conversely, given a space $E$ with an $n$-fold vector bundle structure
over a smooth manifold $M$ as in Definition \ref{n_fold_atlas}, we
define $\mathbb E\colon \square^{\N}\to \Man$ as follows.  Take a
maximal atlas
$\lie A=\{(U_\alpha,\Theta_\alpha, (V_I)_{I\subseteq\nset})\mid \alpha\in
\Lambda\}$
of $E$; in particular $\{U_\alpha\}_{\alpha\in\Lambda}$ is an open covering
of $M$.  For $\alpha,\beta,\gamma\in\Lambda$ we obtain from the
identity
$\Theta_\gamma\circ\Theta_\alpha\inv=\Theta_\gamma\circ
\Theta_\beta\inv\circ\Theta_\beta\circ\Theta_\alpha\inv$
on $\Pi^{-1}(U_\alpha\cap U_\beta\cap U_\gamma)$ the following cocycle
conditions. For $I\subseteq \nset$ and
$\rho=\{I_1,\ldots,I_k\}\in\mathcal P(I)$:
\begin{equation}\label{n-cocycles}
  \begin{split}
    \omega^{\gamma\alpha}_\rho&(p)(v_{I_1},\ldots,v_{I_k})=\\
    &\sum_{\{1,\ldots,k\}=J_1\sqcup\ldots\sqcup J_l}
			\omega^{\gamma\beta}_{\{I_{J_1},\ldots,I_{J_l}\}}(p)
			\Bigl(\omega^{\beta\alpha}_{\{I_j\mid j\in J_1\}}(p)
			\bigl((v_{I_j})_{j\in J_1}\bigr),\ldots,
			\omega^{\beta\alpha}_{\{I_j\mid j\in J_l\}}(p)
			\bigl((v_{I_j})_{j\in J_l}\bigr)\Bigr)\,,
\end{split}
\end{equation}
where $I_{J_m}:=\bigcup_{j\in J_m}I_j$.

We set $\mathbb E(\nset)=E$, $\mathbb E(\emptyset)=M$, 
and more generally for $I\subseteq \nset$, 
\[\mathbb
E(I)=\left.\left(\bigsqcup_{\alpha\in\Lambda}\left(U_\alpha\times\prod_{J\subseteq
  I}V_J\right)\right)\right/\sim \]
with $\sim$ the equivalence relation defined on
$\bigsqcup_{\alpha\in\Lambda}(U_\alpha\times\prod_{J\subseteq I}V_J)$
by 
\[ U_\alpha\times\prod_{J\subseteq I}V_J\quad \ni\quad
\left(p,(v_J)_{J\subseteq I}\right)\quad \sim \quad
\left(q,(w_J)_{J\subseteq I}\right)\quad \in \quad U_\beta\times\prod_{J\subseteq I}V_J
\]
if and only if $p=q$ and
\[ (v_J)_{J\subseteq I}=\left(\sum_{\rho=\{J_1,\ldots,J_k\}\in\mathcal
    P(J)}\omega_{\rho}(p)(w_{J_1},\ldots,w_{J_k})\right)_{J\subseteq I}.
\]
The relations \eqref{n-cocycles} show the symmetry and transitivity of
this relation. As in the construction of a vector bundle from vector
bundle cocycles, one can show that $\mathbb E(I)$ has a unique smooth
manifold structure such that
$\Pi_I\colon \mathbb E(I)\to M$, $\Pi_I[p,(v_I)_{I\subseteq J}]=p$
is a surjective submersion and such that
 the maps 
\[\Theta^I_\alpha\colon \pi_I\left(U_\alpha\times\prod_{J\subseteq
  I}V_J\right)\to U_\alpha\times\prod_{J\subseteq I}V_J, \qquad
[p,(v_I)_{I\subseteq J}]\mapsto (p,(v_I)_{I\subseteq J})
\]
are diffeomorphisms, where $\pi_I\colon \bigsqcup_{\alpha\in\Lambda}(U_\alpha\times\prod_{J\subseteq
  I}V_J)\to \mathbb E(I)$ is the projection to the equivalence classes.

We have then also $\# I$ surjective submersions
\[ p^I_{I\setminus \{i\}}\colon \E(I)\to \E(I\setminus\{i\})\]
for $i\in I$, defined in charts by 
\[ U_\alpha\times \prod_{J\subseteq I}V_J \,\ni\, (p, (v_J)_{J\subseteq
  I})\,\mapsto \, (p, (v_J)_{i\not\in J\subseteq I})\,\in \, U_\alpha\times \prod_{J\subseteq I\setminus\{i\}} V_J
\]
and it is easy to see that $\E(I)$ is a vector bundle over
$\E(I\setminus\{i\})$, and that for $i,j\in I$, 
\[
\begin{tikzcd}
\E(I)\ar[rr,"p^{I}_{I\setminus\{i\}}"] \ar[d,"p^{I}_{I\setminus\{j\}}"] && \E(I\setminus \{i\})\ar[d,"p^{I\setminus \{i\}}_{I\setminus\{i,j\}} "]\\
\E(I\setminus \{j\}) \ar[rr,"p^{I\setminus \{j\}}_{I\setminus\{i,j\}}  "] & &\E(I\setminus \{i,j\})
\end{tikzcd} 
\]
is a double vector bundle,  with obvious local trivialisations given
by the local charts.

\medskip

The constructions above are inverse to each other and we get the
following corollary of our local splitting theorem.
\begin{cor}
Definition \ref{def_n_fold} of an $n$-fold vector bundle as a functor
from the $n$-cube category is equivalent
to Definition \ref{n_fold_atlas} of an $n$-fold vector bundle as a
space with a maximal $n$-fold vector bundle atlas.
\end{cor}

\medskip

Our construction above of an $n$-fold vector bundle atlas on
$\E(\nset)$ from an $n$-fold vector bundle yields an atlas with
simpler changes of charts \eqref{simple_change_of_charts} than the
most general allowed change of charts \eqref{change_of_charts}. This
is due to our choice of a \emph{global} decomposition of the $n$-fold
vector bundle. Choosing different local or global decompositions will
yield an atlas with changes of charts as in \eqref{change_of_charts}.
That the equivalence class of atlases is independent of the choice of
decomposition follows from Proposition \ref{dec_torsor_stato} and
\eqref{statodec}. Two different decompositions will give compatible
charts.

\section{Decompositions of $\infty$-fold vector bundles}\label{sec_infty}

In this section we show how our proof of the existence of linear
decompositions of $n$-fold vector bundles for all $n\in \N$ yields
as well the existence of linear decompositions of $\infty$-fold
vector bundles. We write here $\infVB$ for the category of
$\infty$-fold vector bundles and $\infty$-fold vector bundle
morphisms.

Let $\E$ be an $\infty$-fold vector bundle. Then for each $n\in\N$,
the restriction $\E\circ \iota^\N_n$ defines an $n$-fold vector
bundle, and $\E^n:=\E\circ \iota^\N_n\circ\pi^\N_n$ defines again an
$\infty$-fold vector bundle, given by $\E^n(I)=\E(I\cap\nset)$ for all
finite $I\subseteq \N$. There is a sequence of monomorphisms of
$\infty$-fold vector bundles
\begin{equation}\label{functor_N_to_infVB}
  \E^0\overset{\iota_0^1}{\longrightarrow}\E^1\overset{\iota_1^2}{\longrightarrow}\E^2\overset{\iota_2^3}{\longrightarrow}\ldots
  \end{equation}
  defined by
  $\iota_k^l(I)=0^{I\cap \underline{l}}_{I\cap\underline{k}}$ for $k\leq l\in \N$ and a finite subset $I$ of $\N$;
  remember that $0^I_I=\id_{E_I}$.
  Thus we have a functor $\E^\cdot\colon\N\to \infVB$
  sending an object $n\in \N$ to $\E^n$ and an arrow $m\leq n$ to $\iota_m^n$.
  In the same manner, for each
  $n\in\N$ there is a monomorphism $\iota_n\colon \E^n\rightarrow \E$
  defined by
  $\iota_n(I)=0^{I}_{I\cap\underline{n}}\colon \E^n(I)\to \E(I)$ for
  all finite $I\subseteq \N$.
It is easy to see that
  $\E$ together with the inclusions $\iota_n\colon \E^n\rightarrow \E$
  defines a colimit for \eqref{functor_N_to_infVB} in the category of
  $\infty$-fold vector bundles. 

The inductive nature of the proof
  of Theorem \ref{thm_dec_n_vb} yields the following corollary.
\begin{cor}
  Let $\E$ be an $\infty$-fold vector bundle. Let 
	$\A=(q_{I}\colon A_I\to M)_{I\subseteq \N, \#I<\infty}$ be the family
  of vector bundles over $M$ defined by $A_I=E^I_I$ for $2\leq \#I<\infty$,
  $A_{\{k\}}=E_{\{k\}}$ and $A_\emptyset=\E(\emptyset)=M$. Then
  there exists a sequence of decompositions 
	$\tilde{\S}^n\colon \E^\A\circ \iota^\N_n\rightarrow \E\circ\iota^\N_n$ 
	such that the diagram of $\infty$-fold vector bundles
\begin{equation}\label{infty_com_diag}
	\begin{tikzcd}
	\E^0\ar[r]& \E^1\ar[r] &\E^2\ar[r]&\ldots\\
	(\E^\A)^0\ar[r] \ar[u,"\S^0"] & (\E^\A)^1\ar[r] \ar[u,"\S^1"]&(\E^\A)^2\ar[r] \ar[u,"\S^2"] &\ldots\,,
	\end{tikzcd}
\end{equation}
commutes, where $\S^n(I):=\tilde{\S}^n(I\cap\nset)$ is the morphism
of $\infty$-fold vector bundles induced by $\tilde{\S}^n$.
\end{cor}

Since \eqref{infty_com_diag} commutes, and for each $n$,
$\mathcal S^n$ is an isomorphism, we find that $\E^\A$
together with the morphisms $\tau(n)=\iota^\A_n\circ (\S^n)\inv$ for
all $n$, is also a colimit  for \eqref{functor_N_to_infVB} in the category of
  $\infty$-fold vector bundles. 
Therefore there is a unique isomorphism $\mathcal S\colon \E^\A\rightarrow \E$ such
that $\iota_n\circ \S^n=\mathcal S\circ\iota^\A_n$ for all $n\in \N$.
We get the following theorem.

\begin{thm}
Let $\E$ be an $\infty$-fold vector bundle. Let $\A=(q_{I}\colon A_I\to M)_{I\subseteq \N, \#I<\infty}$ be the family
  of vector bundles over $M$ defined by $A_I=E^I_I$ for $2\leq \#I<\infty$,
  $A_{\{k\}}=E_{\{k\}}$ and $A_\emptyset=\E(\emptyset)=M$. 

Then $\E$ is non-canonically isomorphic 
to the associated decomposed $\infty$-fold vector bundle $\E^\A$.
More precisely, given a tower of decompositions as in \eqref{infty_com_diag}, the decomposition
  $\mathcal S\colon \E^\A\rightarrow \E$ of $\E$ can be uniquely chosen so that
  for each $n\in \N$, $\mathcal S^n\colon (\E^\A)^n\rightarrow \E^n$ satisfies
 \begin{equation}\label{res_inf_mor}
\mathcal S^n(I)=\S(I\cap \nset)\colon (\E^\A)^n(I)=\E^\A(I\cap\nset)\to
  \E^n(I)=\E(I\cap\nset)\end{equation} for all finite $I\subseteq \N$.
\end{thm}

\begin{proof}
The morphism $\mathcal S\colon\E^\A\rightarrow\E$ is explicitly
defined as follows.
Choose a finite subset $I\subseteq \N$. Then there is $n\in\N$ with
$I\subseteq \nset$ and we can set $\mathcal S(I)=\mathcal S^n(I)$. 
The equalities \eqref{res_inf_mor} are now easy to check.
  \end{proof}

\section{Example: triple vector bundles}\label{sec_triple}

In this section, we explain for the convenience of the reader how our
results and considerations in Sections \ref{sec_multiple_def} and
\ref{sec_ex_splitting} read in the case $n=3$. Then we consider \emph{doubly
linear sections} of triple vector bundles, and we explain how they can
be understood -- using linear decompositions -- as horizontal lifts of
pairs of linear sections of the sides double vector bundles.

\subsection{Splittings of triple vector bundles}

Given a triple vector bundle $\E$ we will write in the following 
$T:=\E(\{1,2,3\})$, $D:=\E(\{1,2\})$, $E:=\E(\{2,3\})$, $F:=\E(\{1,3\})$,
$A:=E_{\{1\}}$, $B:=E_{\{2\}}$ and $C:=E_{\{3\}}$.
The triple vector bundle is then a cube of vector bundle structures
\begin{equation}\label{Tcube}
\begin{tikzcd}
  T \ar[rr,"p^T_D"]\ar[rd,"p^T_F"]\ar[dd,"p^T_E"'] & & D \ar[dd,"p^D_B"', near start]\ar[rd,"p^D_A"] & \\
  & F \ar[rr,"p^F_A"', near start, crossing over] & & A \ar[dd,"q_A"] \\
  E \ar[rr,"p^E_B"', near end]\ar[rd,"p^E_C"'] & & B\ar[rd,"q_B"] & \\
  & C \ar[uu, crossing over, leftarrow, "p^F_C", near
  end]\ar[rr,"q_C"] & & M
\end{tikzcd} 
\,,
\end{equation}
where all faces are double vector bundles. 

We will denote the
cores of the double vector bundles $(T;D,E;B)$, $(T;E,F;C)$,
$(T;F,D;A)$ by $L_{DE}$, $L_{EF}$ and $L_{FD}$ and the cores of the
double vector bundles $(D;A,B;M)$, $(E;B,C;M)$, $(F;C,A;M)$ by
$K_{AB}$, $K_{BC}$ and $K_{CA}$, respectively. In the general notation
we would write $E^{\{1,2,3\}}_{\{2,3\}}=:L_{FD}$, $E^{\{1,2,3\}}_{\{1,3\}}=:L_{DE}$ 
and $E^{\{1,2,3\}}_{\{1,2\}}=:L_{EF}$ for the upper cores and 
$E^{\{1,2\}}_{\{1,2\}}=:K_{AB}$, $E^{\{2,3\}}_{\{2,3\}}=:K_{BC}$ 
and $E^{\{1,3\}}_{\{1,3\}}=:K_{CA}$ for the lower cores. 
The \textbf{triple core} of this triple vector bundle is 
$S:=E^{\{1,2,3\}}_{\{1,2,3\}}$, a vector bundle over $M$.

The upper cores $L_{DE}$, $L_{EF}$ and $L_{FD}$ are themselves double 
vector bundles by Theorem \ref{corenvb}. All three have by Lemma \ref{core_of_core}
the core $S$, whereas the sides of 
$L_{DE}$ are given by $K_{CA}$ and $B$, the sides of $L_{EF}$ 
by $K_{AB}$ and $C$, and the sides of $L_{FD}$ by $K_{BC}$ and $A$.

 A decomposition of a triple vector bundle $(T;D,E,F;A,B,C;M)$ as
  above is now an isomorphism of triple vector bundles $\mathcal{S}$ from
  the associated decomposed triple vector bundle as in Example \ref{decTVB}
  to $T$ over decompositions 
  of $D,E$ and $F$ as double vector bundles and inducing the identity
  on $S$. In particular it is over the identities on $A,B$ and $C$,
  and is inducing the identities on $K_{AB},K_{BC}$ and $K_{CA}$.

  A linear splitting of a triple vector bundle $(T;D,E,F;A,B,C;M)$ as
  above is an injective morphism of triple vector bundles $\Sigma$
  from the vacant triple vector bundle
  $(A\times_M B\times_M C;A\times_M B,B\times_M C,C\times_M
  A;A,B,C;M)$
  over linear splittings of the double vector bundles $D,E$ and $F$,
  hence over the identities on $A,B$ and $C$.

  We have proved the following lemma, which is the case $n=3$ of
	Theorem \ref{thm_split_dec_n}. 

\begin{lemma}\label{SplitvsDec}
  A decomposition of a triple vector bundle $T$ is equivalent to a
  linear splitting of $T$ and linear splittings of the three core
  double vector bundles $L_{DE},L_{EF}$ and $L_{FD}$.
\end{lemma}

Note that here, starting from the splittings we get an explicit formula for the decomposition:
$\mathcal{S}(a,b,c,k_{AB},k_{BC},k_{CA},s)$ equals
\begin{equation*}
\begin{split}
  &\Bigl(\bigl(\Sigma(a,b,c)+_D(0^T_{\Sigma^D(a,b)}+_F\Sigma^{L_{FD}}(a,k_{BC}))\bigr)+_F(0^T_{\Sigma^F(a,c)}+_E\Sigma^{L_{EF}}(c,k_{AB}))\Bigr)\\
  & \qquad +_E \Bigl(0^T_{\mathcal S^E(b,c,k_{BC})}+_D\Sigma^{L_{DE}}(b,k_{CA})+_D \bigl({0^{T}_{0_b^D}
	+_{F}\overline{s}}\bigr)\Bigr)\,.
\end{split}
\end{equation*}

\bigskip
Now let us consider the pullback triple vector bundle associated with a triple
  vector bundle. Given double vector bundles $(D,A,B,M)$, $(E,B,C,M)$ and $(F,C,A,M)$,
we consider the set 
\[P=\bigl\{(d,e,f)\in D\times E\times F \mid
  \,\,p^D_A(d)=p^F_A(f),\,\, p^D_B(d)=p^E_B(e),\,\,
  p^E_C(e)=p^F_C(f)\bigr\}\,.\] Then $P$ is a triple
vector bundle, with the obvious projections to $D$, $E$ and $F$ and
the additions defined as follows.  The space $E\times_C
F$ 
has a vector bundle structure
\begin{align*}E\times_C F&\to B\times_M A, \qquad (e,f)\mapsto(p^E_B(e),p^F_A(f))\,,
\end{align*}
with addition $(e_1,f_1)+(e_2,f_2)=(e_1+_Be_2,f_1+_Af_2)$.  Since $D$
is a double vector bundle and so non-canonically split, we have the
surjective submersion $\delta^D\colon D\rightarrow B\times_M A$, given
by $\delta^D(d):=(p^D_B(d),p^D_A(d))$.  We define the vector bundle
$P\to D$ as the pullback vector bundle structure
$(\delta^D)^!(E\times_C F)\rightarrow D$.  We call $P$ the
\textbf{pullback triple vector bundle defined by $D$, $E$ and $F$}
because it fills a cube in a similar manner as the pullback in
category theory fills a square.

We have three short exact sequences of vector bundles over $D$, $E$ and $F$, 
respectively; the one over $D$ reads
\[\begin{tikzcd}[column sep=scriptsize,row sep=small]\label{tvb_core_ses}
0\ar[r] & (\pi^D_M)^!S\ar[r] & T\ar[rrr,"{(\delta^D)^!(p^T_E,p^T_F)}"]& & & P\ar[r] &0\,,\\
\end{tikzcd}
\]
where $\pi^D_M=q_A\circ p^D_A=q_B \circ p^D_B$.
We are now able to state Theorem \ref{thm_dec_n_vb} in the case $n=3$.
\begin{thm}
  Every triple vector bundle is non-canonically isomorphic to a 
decomposed triple vector bundle. 
\end{thm}

\subsection{Splittings, decompositions and horizontal lifts.}
Let us mention first that a decomposition of a double vector bundle is
equivalent to a splitting of the short exact sequences given by its
linear sections.  As we have seen in Section \ref{2-case}, a splitting
$\Sigma\colon A\times_MB\to D$ of $D$ is equivalent to a homomorphism
of $C^\infty(M)$-modules $\sigma_B\colon \Gamma(B)\to\Gamma_A^\ell(D)$
(a \emph{horizontal lift}) which splits this short exact sequence.
The correspondence is given by $\sigma_B(b)(a_m)=\Sigma(a_m,b(m))$ for
all $b\in\Gamma(B)$ and $a_m\in A$.  By symmetry of $\Sigma$ a
horizontal lift $\sigma_B$ is therefore also equivalent to a
horizontal lift $\sigma_A\colon\Gamma(A)\rightarrow \Gamma^\ell_B(D)$,
splitting the sequence
\[0 \rightarrow
\Gamma(\Hom(B,C))\stackrel{\tilde{\cdot}}{\longrightarrow}
\Gamma^\ell_B(D)\stackrel{\pi}{\longrightarrow} \Gamma(A)\rightarrow
0\,.\]

\medskip

In this section, we explain how a splitting of the triple vector
bundle $T$ is equivalent to a ``horizontal lift'' of pairs of linear
sections in $\Gamma_A^\ell(F)\times_{\Gamma(C)}\Gamma_B^\ell(E)$ to
\emph{doubly linear sections} of $T\to D$. Of course, similar results
hold for doubly linear sections of $T\to E$ as lifts of elements of
$\Gamma_C^\ell(F)\times_{\Gamma(A)}\Gamma_B^\ell(D)$, etc.

\begin{mydef}
  A doubly linear section of $T$ over $D$ is a section which is a
  double vector bundle morphism from $(D;A,B;M)$ to $(T;F,E;C)$ over
  some morphisms $\xi\colon A\rightarrow F$, $\eta\colon B\rightarrow E$,
  $c\colon M\rightarrow C$.  The morphisms $\xi$ and $\eta$ are then
  themselves linear sections of the double vector bundles $E$ and $F$
  over the same section of $C$. We denote the set of doubly linear 
  sections of $T$ over $D$ by $\Gamma_D^{\ell^2}(T)$.
\end{mydef}
The space $\Gamma_D^{\ell^2}(T)$ is naturally a $C^\infty(M)$-module: 
for $f\in C^\infty(M)$ and $\xi\in\Gamma^{\ell^2}_D(T)$ doubly linear 
over $\xi_A\in\Gamma_A^\ell(F)$ and $\xi_B\in\Gamma_B^\ell(E)$, the section
$(q_A\circ p^D_A)^*f\cdot \xi$ is doubly linear over $q_A^*f\cdot\xi_A$ 
and $q_B^*f\cdot\xi_B$.

Consider the double vector bundle $S$ with sides $M$ and core $S$:
\[
  \begin{tikzcd}
   S \ar[r,"q_S"] \ar[d,"q_S"'] & M \ar[d,"\id_M"]\\
M \ar[r,"\id_M"] & M
\end{tikzcd} 
\]
As we have seen in Lemma \ref{module_morphisms}, the space
$\operatorname{Mor}_2(D,S)$ of double vector bundle morphisms $D\to S$
is a $C^\infty(M)$-module. It is easy to see that given a
decomposition $A\times_MB\times_MK_{AB}\to D$, we get
$\operatorname{Mor}_2(D,S)\simeq \Gamma(K_{AB}^*\otimes
S)\oplus\Gamma(A^*\otimes B^*\otimes S)$. We have an obvious inclusion
\[ \widetilde{\cdot}\colon \operatorname{Mor}_2(D,S)\hookrightarrow \Gamma_D^{\ell^2}(T),
\]
the images of which are exactly the doubly linear sections that project to the zero sections of $E\to A$ and $F\to B$,
and so to the zero section of $C$.

Both $\Gamma^\ell_A(F)$ and $\Gamma^\ell_B(E)$ project onto
$\Gamma(C)$, thus we can build the pullback
$\Gamma^\ell_A(F)\times_{\Gamma(C)}\Gamma^\ell_B(E)$ which consists of
pairs of linear sections of the respective bundles which are linear
over the same section of $C$.  Now $\Gamma_D^{\ell^2}(T)$ fits into
a short exact sequence of $C^\infty(M)$-modules as in the following
proposition.
\begin{prop}\label{PropSESMod}
  Let $T$ be a triple bundle as in (\ref{Tcube}). We have a short exact sequence of 
	$C^\infty(M)$-modules
  \begin{equation}\label{ModSES}0\to \Mor_2(D,S)\overset{\widetilde\cdot}{\rightarrow}\Gamma_D^{\ell^2}(T)\stackrel{\pi}
{\rightarrow}\Gamma^\ell_A(F)\times_{\Gamma(C)}\Gamma^\ell_B(E)\rightarrow
    0\,.\end{equation}
\end{prop}

\begin{proof}
Injectivity of $\widetilde{\cdot}$ is immediate. To show surjectivity 
of $\pi$, choose a linear splitting $\Sigma^{E,F}$ of the double 
vector bundle $(T;E,F;C)$. Given 
$\xi=(\xi^F,\xi^E)\in \Gamma^\ell_A(F)\times_{\Gamma(C)}\Gamma^\ell_B(E)$
we can then define $\hat{\xi}\in\Gamma_D^{\ell^2}(T)$ by 
$\hat{\xi}(d):=\Sigma^{E,F}(\xi^E(p^D_B(d)),\xi^F(p^D_A(d))$. It is 
easy to see that this is in fact a doubly linear section. Note that 
the map $\hat{\cdot}$ does not define a splitting of the short exact
sequence, as it is not linear over $D$. 

Given any $\phi\in\Mor_2(D,S)$ and $d\in D$ over $a\in A$ and $b\in B$
it is clear that $p^T_E(\tilde{\phi}(d))=\nvbzero{E}{b}$ and
$p^T_F(\tilde{\phi}(d))=\nvbzero{F}{a}$. Thus $\tilde{\phi}$ is linear
over the zero sections of $E\to B$ and $F\to A$ and thus in the kernel
of $\pi$. Conversely, given $\xi \in \Gamma_D^{\ell^2}(T)$ over the
zero sections of $E\to B$ and $F\to A$, we get for any $d\in D$ over
$a\in A$ and $b\in B$ that
$\bigl(\xi(d)-_E\nvbzero{T}{d}\bigr)-_F\nvbzero{T}{{\nvbzero{D}{b}}}$
projects to zero in all directions and thus defines an element
$\phi(d)$ of the triple core $S$. It is easy to check that this
assignment defines a morphism $\phi\in\Mor_2(D,S)$. Then
$\xi=\tilde{\phi}$ and the sequence is exact.
\end{proof}

\begin{prop}\label{eq_lift_dec}
  A decomposition of a triple vector bundle $T$ as in (\ref{Tcube}) is
  equivalent to linear splittings of the double vector bundles $D$,
  $E$, $F$, $L_{DE}$ and $L_{FD}$ and a \textbf{horizontal lift}, that
  is a splitting
  $\sigma\colon
  \Gamma^\ell_A(F)\times_{\Gamma(C)}\Gamma^\ell_B(E)\to
  \Gamma_D^{\ell^2}(T) $ of the short exact sequence (\ref{ModSES})
  that is compatible with the splittings of the double vector bundles
  in the sense that for all $d\in D$  we have 
  $\sigma(\widetilde{\phi^F},0^E_B)(d)=
  \nvbzero{T}{d}+_E\Sigma^{L_{DE}}\bigl(p^D_B(d),\phi^F(p^D_A(d))\bigr)$
  for all $\phi^F\in\Gamma(\Hom(A,K_{CA}))$ and
  $\sigma(0^F_A,\widetilde{\phi^E})(d)=
  \nvbzero{T}{d}+_F\Sigma^{L_{FD}}\bigl(p^D_A(d),\phi^E(p^D_B(d))\bigr)$
  for all $\phi^E\in\Gamma(\Hom(B,K_{BC}))$.
\end{prop}

\begin{proof}
	A given decomposition $\S$ of $T$ induces decompositions of all the 
	double vector bundles by definition. These are equivalent to linear 
	splittings and horizontal lifts 
	$\sigma^E_C\colon\Gamma(C)\to\Gamma^\ell_B(E)$ and 
	$\sigma^F_C\colon\Gamma(C)\to\Gamma^\ell_A(F)$. Now any two linear sections
	$\xi^E\in\Gamma^\ell_B(E)$ and $\xi^F\in\Gamma^\ell_A(F)$ over the same 
	$c\in\Gamma(C)$ can be written as $\xi^E=\sigma^E_C(c)+\widetilde{\phi^E}$ 
	and $\xi^F=\sigma^F_C(c)+\widetilde{\phi^F}$ for some 
	$\phi^E\in\Gamma(B^*\otimes K_{BC})$ and $\phi^F\in\Gamma(A^*\otimes K_{AC})$. 
	We define a horizontal lift by 
	\[\sigma\bigl(\xi^E,\xi^F\bigr)\bigl(\S^D(a_m,b_m,k_m)\bigr):=
	\S\bigl(a_m,b_m,c(m),k_m,\phi^F(a_m),\phi^E(b_m),0^S_m\bigr)\,.
	\]
	It is easy to check that this lift satisfies the additional compatibility
	conditions. 

	Conversely, given linear splittings of the double vector bundles $D$, $E$,
	$F$, $L_{DE}$, $L_{FD}$ and a horizontal lift $\sigma$ satisfying the
	extra condition, we first define a linear splitting
	$\Sigma^{L_{EF}}\colon C\times_M K_{AB}\to L_{EF}$ 
	by 
	$\Sigma^{L_{EF}}(c_m,k_{AB}):=\sigma(\sigma_C^F(c),\sigma_C^E(c))(k_{AB})$
        for any section $c$ of $C\to M$ with $c(m)=c_m$, and 
	where we view $K_{AB}$ as a subset of $D$. 
	Then we define a linear splitting of $T$ by 
	\[
	\Sigma(a_m,b_m,c_m):=\sigma\bigl(\sigma^E_C(c),\sigma^F_C(c)\bigr)(\Sigma^D(a_m,b_m))
	\]
	where $c\in\Gamma(C)$ is any section such that $c(m)=c_m$.  Together
	with Lemma \ref{SplitvsDec} this gives a decomposition of $T$.
	
	Straightforward computations show that these two constructions are 
	indeed inverse to each other and we get the desired equivalence. 	
\end{proof}

\medskip

The analogon of Proposition \ref{PropSESMod} for general $n$ is easy
to write down and prove \cite{Heuer19}, but Proposition
\ref{eq_lift_dec} becomes highly technical for increasing $n$. It is
relatively easy to see that a horizontal lift defines a linear
splitting of the $n$-fold vector bundle, and conversely that a
decomposition of an $n$-fold vector bundle defines a horizontal
lift. However, as the additional conditions in Proposition
\ref{eq_lift_dec} and in Theorem \ref{thm_split_dec_n} suggest, the
formulation of equivalent constructions is not straightforward. This
is work in preparation \cite{Heuer19}.


\def\cprime{$'$} \def\polhk#1{\setbox0=\hbox{#1}{\ooalign{\hidewidth
  \lower1.5ex\hbox{`}\hidewidth\crcr\unhbox0}}} \def\cprime{$'$}
  \def\cprime{$'$}

\end{document}